\documentclass[11pt]{amsart}
\usepackage{amssymb}
\usepackage{amsmath,amsthm,geometry,enumerate}
\usepackage{psfrag}
\usepackage{amscd}
\usepackage{xspace}
\usepackage[all]{xy}
\usepackage{tikz}
\usetikzlibrary{trees}

\tikzstyle{every picture}=[scale=.5,inner sep=10]
\tikzstyle{every node}=[draw,circle,inner sep=1]

\normalfont
\numberwithin{equation}{section} 

\let\emptyset\varnothing

\DeclareMathOperator{\rk}{rk}

\DeclareMathOperator{\Aut}{Aut}

\DeclareMathOperator{\id}{id}

\DeclareMathOperator{\codim}{codim}
\DeclareMathOperator{\coker}{Coker}

\newtheorem{conjecture}[equation]{Conjecture}
\newtheorem{construction}[equation]{Construction}
\newtheorem{proposition}[equation]{Proposition}
\newtheorem{propositionstar}[equation]{Proposition$^*$}
\newtheorem{theorem}[equation]{Theorem}
\newtheorem{theoremstar}[equation]{Theorem$^*$}
\newtheorem{corollary}[equation]{Corollary}
\newtheorem{corollarystar}[equation]{Corollary$^*$}
\newtheorem{lemma}[equation]{Lemma}

\theoremstyle{definition}
\newtheorem{definition}[equation]{Definition}
\newtheorem{definitionstar}[equation]{Definition$^*$}

\newtheorem{remark}[equation]{\textbf{Remark}}
\newtheorem{notation}[equation]{\textbf{Notation}}

\newtheorem{example}[equation]{\textbf{Example}}

\newcommand{\chom}{{{\mathcal{H}}om}}
\newcommand{\cext}{{{\mathcal{E}}xt}}

\newcommand{\mb}[2]{$\overline{\mathcal{M}}_{#1,#2}$}

\newcommand{\mgn}{$\mathcal{M}_{g,n}$}
\newcommand{\mbgn}{$\overline{\mathcal{M}}_{g,n}$}
\newcommand{\mun}{$\mathcal{M}_{1,n}$}
\newcommand{\mbun}{$\overline{\mathcal{M}}_{1,n}$}

\newcommand{\mdnrt}{$\mathcal{M}_{2,n}^{rt}$}

\newcommand{\mbdn}{$\overline{\mathcal{M}}_{2,n}$}

\newcommand{\virg}[1]{\textquotedblleft#1\textquotedblright}

\begin{document}

\title{\textbf{The Chen--Ruan cohomology of moduli of curves of genus $2$ with marked points}}
\subjclass{14H10 14N35}

\author{Nicola Pagani}

\begin{abstract}In this work we describe the Chen--Ruan cohomology of the moduli stack of smooth and stable genus $2$ curves with marked points. In the first half of the paper we compute the additive structure of the Chen--Ruan cohomology ring for the moduli stack of stable $n$-pointed genus $2$ curves, describing it as a rationally graded vector space. In the second part we give generators for the even Chen--Ruan cohomology ring as an algebra on the ordinary cohomology.
\end{abstract}
\maketitle

\setcounter{tocdepth}{2}
\tableofcontents


\section {Introduction}

Chen--Ruan cohomology, also known as orbifold cohomology, was recently developed as part of the project of extending Gromov--Witten theory to orbifolds. The first definition was given in the language of differential geometry in the seminal paper \cite{chenruan}, making precise some ideas from theoretical and mathematical physics that appeared in the nineties. In the algebraic category, it was introduced by Abramovich--Graber--Vistoli in the papers \cite{agv1} and \cite{agv2}, and it is called stringy Chow ring. This cohomology ring recovers as a subalgebra the ordinary rational cohomology ring of the topological space that underlies the orbifold. As a vector space, the Chen--Ruan cohomology of $X$ is the cohomology of the inertia stack of $X$. If $X$ is an orbifold, its \emph{inertia stack} $I(X)$ is constructed as the disjoint union, for $g$ in the stabilizer of some point $x$ of $X$, of the locus stabilized by $g$ in $X$ (see Definition \ref{sectiondefinertia}). For example, the orbifold $X$ itself appears as a connected component of $I(X)$: indeed it is the locus fixed by the identity, which is trivially contained in the stabilizer group of every point. All the other connected components of the inertia $I(X)$ are usually called \emph{twisted sectors}. In this paper we work in the algebraic category, and whenever the word \virg{orbifold} is mentioned, it stands for smooth Deligne--Mumford stack. Moreover we only deal with cohomology and Chow groups with coefficients in the field $\mathbb{Q}$ of rational numbers.

The moduli spaces $\mathcal{M}_{g,n}$ and their Deligne--Mumford compactifications $\overline{\mathcal{M}}_{g,n}$ are among the very first orbifolds studied in algebraic geometry. In particular, the study of their cohomology rings has been pursued with techniques coming from algebraic geometry and algebraic topology, starting from the eighties with the works of Harer \cite{harer}, and Mumford \cite{mumford}. In the last twenty years, our knowledge about the structure of such rings has tremendously increased.

The Chen--Ruan cohomology of the moduli spaces of curves with marked points has been studied so far for the case of genus $1$ in \cite{pagani1}. The stringy Chow ring of $\mathcal{M}_2$ has been studied by Spencer \cite{spencer}, although the result is slightly incomplete as two twisted sectors are missing from his picture. Some computations in the stringy Chow ring of $\overline{\mathcal{M}}_2$ have been also approached in \cite{spencer2}. The Chen--Ruan cohomology of moduli of smooth hyperelliptic curves is additively known after \cite{paganihyper}.

In this paper, we investigate the Chen--Ruan cohomology ring of the moduli space of smooth genus $2$ curves with marked points, and of the moduli space of stable genus $2$ curves with marked points. The problem splits naturally into two parts. The first part is the study of the inertia $I(\overline{\mathcal{M}}_{2,n})$: we work this out in Sections \ref{Inertia}, \ref{cohomology} and \ref{grading}. The second part is the investigation of the Chen--Ruan cohomology of $\overline{\mathcal{M}}_{2,n}$ as a ring, which consists of computing a virtual fundamental class and then computing pull-back of classes among moduli spaces: this is studied in Sections \ref{stringyproduct} and \ref{algebra}. In this picture, it is not more difficult to state and prove some of the intermediate results more generally for all $g$ and $n$. A very convenient step is usually to find results for $\mathcal{M}_{g,n}^{rt}$, the partial compactification of $\mathcal{M}_{g,n}$ made of stable curves with a smooth irreducible component of genus $g$ (and, possibly, rational tails).

The main results of this paper are the following. In sections \ref{inertia} and \ref{rationaltails} we solve the problem of identifying the connected components of the inertia stack of $\mathcal{M}_2$, and of $\mathcal{M}_{g,n}$ and $\mathcal{M}_{g,n}^{rt}$ for $n \geq 1$ or $g =2$: see Definition \ref{2admissible} and Propositions \ref{instack2}
, \ref{twratcor}. The twisted sectors of $\mathcal{M}_{g,n}$, $n \geq1$ or $g=2$ are given a modular interpretation in terms of moduli of smooth pointed curves with lower genus $g'$, see Definition \ref{settoretwistato}\footnote{The case of the twisted sectors of $\mathcal{M}_g$, $g\geq3$ is more delicate and requires a further analysis.}. In Theorem \ref{samuel2stab}, we compute the dimensions of the Chen--Ruan cohomologies of $\overline{\mathcal{M}}_{2,n}$ in terms of the dimensions of the ordinary cohomologies of $\overline{\mathcal{M}}_{g,n}$ for $g \leq 2$. In Theorem \ref{poincare2} we write down explicitly the orbifold Poincar\'e polynomial for $\overline{\mathcal{M}}_2$. In Theorem \ref{positivo} we see how the algebra $H^*_{CR}(\overline{\mathcal{M}}_{2,n})$ can be generated as an algebra on $H^*(\overline{\mathcal{M}}_{2,n})$ by the fundamental classes of the twisted sectors and suitably defined classes $\mathcal{S}^{I_1,I_2}$, for $I_1, I_2$ that vary among all the non-empty partitions of $[n]$. Finally, in Definition \ref{chenruantautolo} we advance a proposal for an orbifold tautological ring, a well-behaved subring of the even Chen--Ruan cohomology, for stable genus $2$ curves (for stable genus $1$ curves, this was done in \cite{pagani1}).

The main techniques used in this paper are: abelian cyclic coverings of curves (see \cite{pardini}, \cite{cornalba}), their moduli spaces and their compactifications by means of admissible coverings, and some elementary deformation theory of nodal curves. In the last section we take advantage of the existing technology of intersection theory in the tautological ring of moduli spaces of curves.

To conclude this introduction, we make a couple of observations that relate the paper to the general project of studying the Chen--Ruan cohomology of $\overline{\mathcal{M}}_{g,n}$ for general $g$ and $n$. The case of genus $2$ can be seen as the simplest case for which non-trivial phenomena occur. The issues that make the description more difficult than in genus $1$ are: the automorphism groups of genus $2$ curves can be non-abelian, the twisted sectors of moduli of genus $2$ curves are not necessarily closed substacks of the original moduli stack, not all twisted sectors can be constructed by adding rational tails to the compactification of the twisted sectors of $\mathcal{M}_{2,n}$, the cohomology of a twisted sector is not always generated by its divisor classes and finally, the Chen--Ruan cohomology ring is not generated as an $H^*$-algebra by the fundamental classes of the twisted sectors. 

We believe that the solutions that we propose to overcome these difficulties in genus $2$ can potentially be used for general genus $g$, once an explicit stratification of the moduli space by automorphism groups is given (the latter task seems to be very hard for higher genus). Finally, we expect that the simple description of multiplicative generation that we can give for genera $1$ and $2$ would become far too complicated for higher genus. 


\subsection{Summary of the results}

Here we provide more specific details on the results that we obtain in this paper.

In \underline{Section \ref{Inertia}}, we study the inertia stack of $\overline{\mathcal{M}}_{2,n}$. In Section \ref{inertia} we deal with the inertia stack of $\mathcal{M}_{g,n}$ and its compactification by means of admissible coverings. For $n\geq1$ or $g=2$, we prove that the twisted sectors of $\mathcal{M}_{g,n}$ correspond to moduli of cyclic coverings in Proposition \ref{instack2}, and in Corollary \ref{compsmooth} that the compactifications of the latter correspond to the twisted sectors of $\overline{\mathcal{M}}_{g,n}$ that do not come from the boundary (see \cite[Section 2.b]{paganitommasi} for an extension of this description to the case of $\mathcal{M}_g$, $g \geq 3$). 

The cohomology of the cyclic coverings that cover curves of genus $0$ is known, as observed in Corollary \ref{corollariocadman}, so we focus on studying the geometry of the moduli space $II$ of bielliptic genus $2$ curves with a distinguished bielliptic involution (we call $II_1$ the same moduli space with an ordering of the ramification points). In Propositions \ref{duezero}, \ref{aggiunta} and in Corollary \ref{aggiunta2} we compute the cohomology groups of these moduli spaces of bielliptic curves and of their compactifications. We conclude the section by proving, in Theorem \ref{generatodivisori}, that the cohomology of the twisted sectors of $\overline{\mathcal{M}}_{2,n}$ that do not come from the boundary coincides with the Chow group, and that it is multiplicatively generated by the divisors.

In Section \ref{rationaltails}, we study what happens when rational tails are added to the twisted sectors of $\overline{\mathcal{M}}_{g,n}$. In particular, in Proposition \ref{twratcor}, $I(\mathcal{M}_{g,n}^{rt})$ is explicitly described in terms of $I(\mathcal{M}_{g,k})$ for $k \leq \min(2g+2,n)$. 

In Section \ref{dalbordo}, we solve the combinatorial problem of identifying the twisted sectors coming from the boundary of $\overline{\mathcal{M}}_{2,n}$. We give pictures of these twisted sectors using stable graphs, and some twisted sectors of moduli of curves of genus lower than $2$.

In \underline{Section \ref{cohomology}} we are then able, as a consequence of the computations of the previous section, to give the generating series of the dimensions of the Chen--Ruan cohomology of $\mathcal{M}_{2,n}^{rt}$ in Theorem \ref{samuel2thm} and of $\overline{\mathcal{M}}_{2,n}$ in Theorem \ref{samuel2stabthm}. These generating series are given in terms of the generating series for the dimensions of the respective ordinary cohomologies.

In \underline{Section \ref{grading}} we investigate the Chen--Ruan cohomology of $\mathcal{M}_{g,n}$, $\mathcal{M}_{g,n}^{rt}$ and $\overline{\mathcal{M}}_{2,n}$ as a graded vector space. Assuming knowledge of the age of the twisted sectors of $\mathcal{M}_g$ (which we know for $g=2$), in Lemma \ref{etalisci} (resp. Corollary \ref{agerational}), we give a formula to compute the age of the twisted sectors of $\mathcal{M}_{g,n}$ (resp. $\mathcal{M}_{g,n}^{rt}$). We explicitly write the orbifold Poincar\'e polynomial for $\mathcal{M}_{2,n}$ in Theorem \ref{poincare2smooth} and for $\overline{\mathcal{M}}_2$ in Theorem \ref{poincare2}. 

In \underline{Section \ref{stringyproduct}} we study the Chen--Ruan cohomology as a ring. This amounts to computing the virtual fundamental class for the moduli space of stable maps of degree $0$ having target $\overline{\mathcal{M}}_{2,n}$, from a source curve of genus $0$ with $3$ marked points. This virtual class can be expressed as the top Chern class of a certain orbifold excess intersection bundle. 

In Section \ref{secondinertia} we study, in Propositions \ref{terzoterzo}, \ref{quartoquarto}, and in Corollaries \ref{terzoterzo1}, \ref{quartoquarto1}, some of the components of the second inertia stack of $\overline{\mathcal{M}}_{2,n}$. It will be clear in the subsequent section that these components are all those where the virtual class is not $1$ nor $0$.

In Section \ref{excessintersection} we compute all these virtual classes and we express them in terms of $\psi$-classes on moduli stacks of stable genus $0$ or genus $1$ pointed curves. To pursue this, we first study, in Lemma \ref{lemmadecomp}, the normal bundle of the double twisted sectors of $\mathcal{M}_{g,n}$ to $\mathcal{M}_{g,n}$ itself, as a representation of the group that defines the twisted sector. In Corollary \ref{corollariozero} we prove that the virtual classes over the double twisted sectors that are obtained from zero-dimensional twisted sectors of $\mathcal{M}_g$ by adding marked points are always $0$ or $1$. The virtual class is computed in the remaining cases in Propositions \ref{treuno}, \ref{trequattro}, and Corollary \ref{tredue} for the double twisted sectors with smooth general element, and in Proposition \ref{tretre} by reduction to lower genus for the double twisted sectors whose general element does not contain a smooth irreducible genus $2$ curve. 

Finally, in \underline{Section \ref{algebra}} we see how the Chen--Ruan cohomology ring $H^{ev}_{CR}(\overline{\mathcal{M}}_{2,n})$ can be generated as an algebra on $H^{ev}(\overline{\mathcal{M}}_{2,n})$. 
To obtain this, we have to analyze the pull-back in cohomology of the natural map from the twisted sectors to $\overline{\mathcal{M}}_{2,n}$. In Proposition \ref{suriettivita1} we prove that this pull-back is surjective for all the twisted sectors with smooth general element apart from the moduli of bielliptic curves, and for all the twisted sectors coming from the boundary, apart from those whose general element has an irreducible component of genus $1$. The surjectivity of the pull-back map in the latter case is obtained in Corollary \ref{corollariogetzler} for the even cohomology. The proof of Corollary \ref{corollariogetzler} is conditional to Getzler's conjecture \ref{getzlerremark}.

So it remains to study a little more of the geometry of the compactified moduli spaces of bielliptic curves. In Proposition \ref{generazione}, we identify three geometric generators for the Picard groups of these moduli spaces, and then in Proposition \ref{corollariocoker} we see that one of these three classes, which we call $\mathcal{S}$, can be chosen as a generator for the one-dimensional cokernel of the pull-back map. We can thus conclude, conditionally Getzler's conjecture \ref{getzlerremark}, the main Theorem \ref{positivo} on the generation of $H^{ev}_{CR}(\overline{\mathcal{M}}_{2,n})$ as an $H^{ev}(\overline{\mathcal{M}}_{2,n})$-algebra, by marking the ramification points of $\mathcal{S}$ and gluing rational tails. Then in Definition \ref{chenruantautolo} we construct a subring of the even cohomology, which we call orbifold tautological ring.

In \underline{Appendix A}, we collect some useful facts on the inertia stacks of $[\mathcal{M}_{0,n}/S_2]$ and $[\mathcal{M}_{1,n}/S_2]$ that are used in Section \ref{dalbordo}.


\section {The inertia stacks}
\label{Inertia}


\subsection{Definition of the inertia stack}

\label{sectiondefinertia}
In this section we recollect some basic notions concerning the inertia stack. For a more detailed study of this topic, we address the reader to \cite[Section 3]{agv2}. We introduce the following natural stack associated to a stack $X$, which points to where $X$ fails to be an algebraic space.

\begin{definition} \label{definertia} (See \cite[4.4]{agv1}, \cite[Definition 3.1.1]{agv2}.) Let $X$ be an algebraic stack. The \emph{inertia stack} $I(X)$ of $X$ is defined as:
$$
I(X) := \coprod_{N \in \mathbb{N}} I_N(X),
$$
where $I_N(X)(S)$ is the following groupoid.
\begin{enumerate}
\item The objects are pairs $(\xi, \alpha)$, where $\xi$ is an object of $X$ over $S$, and $\alpha: \mu_N \to \Aut(\xi)$ is an injective homomorphism,
\item The morphisms are the morphisms $g: \xi \to \xi'$ of the groupoid $X(S)$, such that $g \cdot \alpha(1)= \alpha'(1) \cdot g$.  
\end{enumerate}
We also define $I_{TW}(X):= \coprod_{N \in \mathbb{N}, N \neq 1}I_N(X)$, in such a way that: $$I(X)=I_1(X) \coprod I_{TW}(X).$$ The connected components of $I_{TW}(X)$ are called \emph{twisted sectors} of the inertia stack of $X$, or also twisted sectors of $X$. The inertia stack comes with a natural forgetful map $f:I(X) \to X$.

\end {definition}

The Chen--Ruan cohomology group, and respectively the stringy Chow group of a stack $X$ (as first defined in \cite{chenruan} and \cite{agv1}) are simply the ordinary rational cohomology group, resp. the rational Chow group of the inertia stack $I(X)$. The Chen--Ruan cohomology is then given an unconventional grading over the rational numbers, as we shall see in Section \ref{grading}. For this section, we give a preliminary definition.

\begin{definition} \label{defcoomorb1} Let $X$ be a Deligne--Mumford stack. The \emph{Chen--Ruan cohomology} of $X$ is defined as a vector space as:
$$
H^*_{CR}(X):= H^*(I(X), \mathbb{Q}).
$$
The \emph{stringy Chow group} of $X$ is defined as a vector space as:
$$
A^*_{st}(X):=A^*_{\mathbb{Q}}(I(X)).
$$
\end{definition}

We observe that, by our very definition, $I_N(X)$ is an open and closed substack of $I(X)$, but it rarely happens  that it is connected. One special case is when $N$ is equal to $1$: in this case the map $f$ restricted to $I_1(X)$ induces an isomorphism of the latter with $X$. The connected component $I_1(X)$ will be referred to as the \emph{untwisted sector}.

 We also observe that given a choice of a primitive generator of $\mu_N$, we obtain an isomorphism of $I(X)$ with $I'(X)$, where the latter is defined as the ($2$-)fiber product $X \times_{X \times X} X$ where both morphisms $X \rightarrow X \times X$ are the diagonals.

\begin{remark} \label{mappaiota} There is an involution $\iota: I_N(X) \to I_N(X)$, which is induced by the map $\iota': \mu_N \to \mu_N$, that is $\iota'(\zeta):= \zeta^{-1}$.
\end{remark}

\begin{proposition} \label{liscezza1} (See \cite[Corollary 3.1.4]{agv2}.) Let $X$ be a smooth algebraic stack. Then the stacks $I_N(X)$ (and therefore $I(X)$) are smooth.
\end{proposition}

Now we study the behaviour of the inertia stack under arbitrary morphisms of stacks. 

\begin {definition}\label{pullinertia} Let $f: X \to Y$ be a morphism of stacks. We define $f^*(I(Y))$ as the stack that makes the following diagram $2$-Cartesian $$
\xymatrix{f^*(I(Y)) \ar[r]^{I(f)} \ar[d] \ar@{}|{\square}[dr] & I(Y) \ar[d] \\
X \ar[r]^{f} & Y 
}$$
and $I(f)$ as the map that lifts $f$ in the diagram.
\end {definition}
\noindent There is an induced map that we call $I'(f)$, which maps $I(X) \to f^*(I(Y))$.
There is a necessary and sufficient condition for $I(X)$ to coincide with $f^*(I(Y))$.

\begin {proposition} \label{strongrap}(folklore) Let $f: X \to Y$ be a morphism of stacks. Then $I(X)$ coincides with $f^*(I(Y))$ if and only if the map $f$ induces an isomorphism on the automorphism groups of the objects.
\end {proposition}

We now present our strategy to describe the twisted sectors of $\overline{\mathcal{M}}_{2,n}$, the moduli stack of stable curves of genus $2$. We consider the filtration: 
$$
\mathcal{M}_{g,n} \subset \mathcal{M}_{g,n}^{rt} \subset \overline{\mathcal{M}}_{g,n},
$$
where $\mathcal{M}_{g,n}^{rt}$ is the moduli stack of stable curves of genus $g$ and $n$ marked points \emph{with rational tails} whose objects are stable genus $g$, $n$-pointed curves with one irreducible component that is smooth of genus $g$. As we showed in \cite[Theorem 1.1]{pagani1} (\emph{cf.} also \cite[Corollary 5.19]{paganitesi}), the inertia stack of $\mathcal{M}_{1,n}^{rt}$ is dense in the inertia stack of \mbun. This is no longer true in the genus $2$ case. As we shall see in Section \ref{dalbordo}, there are several twisted sectors of $\overline{\mathcal{M}}_{2,n}$ whose canonical image in $\overline{\mathcal{M}}_{2,n}$ is strictly contained in $\overline{\mathcal{M}}_{2,n} \setminus \mathcal{M}_{2,n}^{rt}$.

 
\subsection {The inertia stack of moduli of smooth curves of genus $2$}
\label{inertia}

In this section, we study the geometry of the connected components of the inertia stack of $\mathcal{M}_{g,n}$, focusing in particular on the case $g=2$.

  The approach we follow is inspired by Fantechi \cite{fantechi}. This approach gives a technology to identify the twisted sectors of the inertia stack of $\mathcal{M}_{g,n}$, by means of some discrete numerical data. The enumeration of the twisted sectors is thus reduced to the combinatorial problem of finding all the admissible $(g,n)$-data. This approach also gives a modular-theoretic interpretation of the twisted sectors of $\mathcal{M}_{g,n}$, in terms of moduli spaces of smooth pointed curves of lower genus $g'$. We are able to accomplish this program in the cases when $n\geq1$ or $g=2$; our approach can be extended to $\mathcal{M}_g$ for $g>2$ as soon as one can prove the irreducibility of the moduli spaces of cyclic covers with given numerical data and total space a connected curve of genus $g$.

We study 
\label{inertia2} the moduli stacks of ramified cyclic $\mathbb{Z}_N$-coverings of curves of a given genus $g'$ with fixed branching datum. We use for this the description due to Pardini of cyclic abelian coverings, and in particular \cite[Proposition 2.1]{pardini}. A key role in her description is played by the set of pairs $(H, \psi)$, where $H$ is a subgroup of $\mathbb{Z}_N$ and $\psi$ is an injective character of $H$. Since we work over the complex numbers, we can identify this set with the set $\{1, \ldots, N-1\}$, via the following bijection:
\begin{equation} \label{quasicanonica}
\mu :\{1, \ldots, N-1 \} \to \{(H, \psi)| \ \{\id\} \neq H<G,  \ \psi \in H^{\vee} \textrm{ is an injective character} \}
\end{equation}
given by $\mu(m)=\left( \langle m \rangle < \mathbb{Z}_N, m \to  e^{{2 \pi i} \frac{\gcd(m,N)}{N} }\right)$. 

\begin{definition} \label{2admissible} A $(g,n)$-admissible datum will be a tuple $A=(g',N,d_1,\ldots,d_{N-1}, a_1,\ldots, a_{N-1})$ of non-negative integers with $N \geq 2$ and $g'\geq 0$, satisfying the following conditions:
\begin{enumerate}
\item the Riemann-Hurwitz formula holds: $$2g-2=N (2 g'-2)+  \left(\sum d_i \ \gcd(i,N)\left(\frac{N}{\gcd(i,N)}-1\right)  \right);$$
\item $\sum i \ d_i= 0 \mod N$, this is formula \cite[2.14]{pardini} written in the case of cyclic coverings, after the identification \ref{quasicanonica};
\item $\sum a_i=n$;
\item for all $i$, $a_i \leq d_i$;
\item if $\gcd(i,N) \neq 1$, then $a_i=0$.
\end{enumerate}
\end{definition}
\noindent The points $3,4,5$ correspond to the choice of $n$ marked points among the points of total branching (or, equivalently, of total ramification) for the covering. Note that the set of $(g,n)$-admissible data $A$ is finite, thanks to conditions $1$ and $3$.

To each $(g,n)$-admissible datum, we associate the integer $d= \sum d_i$, a disjoint union decomposition $\{1,\ldots, d\}= \coprod_{i=1}^{N-1} J_i$, where:
$$
J_i:= \left\{j | \ \sum_{l<i} d_l < j \leq \sum_{l \leq i} d_l \right\}
$$
and subsets $A_i$ consisting of the first $a_i$ elements of each $J_i$. Moreover, we define $S_A$ to be the subgroup of $S_d$ (the symmetric group on $d$ elements):
\begin{equation} \label{sa} S_A := \left\{ \sigma| \ \forall i \ \sigma (J_i)=(J_i), \forall a \in \sqcup A_i, \sigma(a)=a \right\}.\end{equation}
We now come to the central object of our construction.

\begin{definition} (See also \cite[Definitions 4.22, 4.35]{paganitesi} for an expanded version of this definition.) \label{settoretwistato} Let $A$ be a $(g,n)$-admissible datum. We define the stack $\mathcal{M}_A$ to be the stack parametrizing tuples $(C',D_1, \ldots, D_{N-1},p_1, \ldots, p_n, L, \phi)$, where $C'$ is a smooth genus $g'$ curve, $D_i$ are disjoint smooth effective divisors in $C'$ of degree $d_i$, $L$ is a line bundle on $C'$, $\phi:L^{\otimes N} \to \mathcal{O}_{C'} (\sum  i D_i)$ is an isomorphism, and the $p_j$ are pairwise distinct points of $C'$. The point $p_j$ belongs to $D_i$ for the unique $i$ such that $\sum_{l <i} a_l <j \leq \sum_{l \leq i} a_l$ and $\gcd(i,N)=1$.
\end{definition}

\noindent  Let $\overline{\mathcal{M}}_{g',d}(B \mu_N)$ be the moduli stack of stable maps with value in $B \mu_N$ defined in \cite{abvis2} (see also \cite{acv}, where $\overline{\mathcal{M}}_{g',d}(B \mu_N)$ is called $\mathcal{B}_{g',d}(\mu_N)$). Let moreover $\mathcal{M}_{g',d}(B \mu_N)$ be the open substack consisting of smooth source curves. The stack $\mathcal{M}_A$ we have just defined, is an open and closed substack of $[\mathcal{M}_{g',d}(B \mu_N)/S_A]$ prescribed by the assignment of $d_i$ points of ramification type $i$ (under convention \ref{quasicanonica}).
 
 \begin{example} Let us consider the $(2,1)$-admissible datum $(0,6;d_1=2,d_4=1;a_1=1)$. This corresponds to the moduli space of $\mathbb{Z}_6$-coverings $f:C \to C'$ with $g(C')=0$ and three branch points $s_1,s_2,s_3$. The first two points are of total branching and $s_1$, or equivalently $f^{-1}(s_1)$, is marked: thus $p_1:=s_1$. The generator $\overline{1} \in \mathbb{Z}_6$ acts on the cotangent spaces $T^{\vee}f^{-1}(s_{1,2})$ by multiplication by $e^{\frac{2 \pi i}{6}}$. There are two fibers in $f^{-1}(s_3)$, and $\overline{4} \in \mathbb{Z}_6$ acts on each of the two cotangent spaces over $s_3$ as the multiplication by $e^{\frac{2 \pi i}{3}}$. According to the notation that will be established in Notation \ref{notazionemg}, this moduli space corresponds to the twisted sector called $V.1_1$.
 \end{example}
 
\begin{proposition} \label{connessione} Let $A$ be a $(g,n)$-admissible datum with $n \geq1$ or $g=2$, then the stack $\mathcal{M}_A$ is connected.
\end{proposition}
\begin{proof}
Cornalba has proved this in \cite[p.3]{cornalba} when the order $N$ of the cyclic group is a prime number. His proof extends to the case of general $N$, provided that there is a point of total ramification. In terms of a $(g,n)$-admissible datum $A$ this translates into the fact that there exists $i$ such that $\gcd(i,N)=1$ and $d_i>0$. Another set of cases where the connectedness of $\mathcal{M}_A$ is easy to prove is when $g'$ equals $0$: in this case $\mathcal{M}_A \cong [\mathcal{M}_{0,d}/S_A]$.

Now if $n\geq1$, by the very definition of a $(g,n)$-admissible datum, we are ruling out all the moduli spaces that do not have a point of total ramification. If $g$ equals $2$, then $g'$ can be equal to $0$ or $1$. In the second case, the Riemann-Hurwitz formula forces $N$ to equal $2$, and so this case is covered by Cornalba's proof.
\end{proof}

\begin{proposition} \label{instack2} Let $A$ be a $(g,n)$-admissible datum such that $n \geq 1$ or $g=2$. Then, any assignment $\alpha: \{1,\ldots,n\} \to \mathbb{Z}_N^*$ such that $|\alpha^{-1}(i)|=a_i$ induces an isomorphism of $\mathcal{M}_A$ with a connected component of $I_N(\mathcal{M}_{g,n})$. Conversely, to any connected component $X$ of $I_N(\mathcal{M}_{g,n})$ such that $n \geq 1$ or $g=2$, one can associate a $(g,n)$-admissible datum $A$ such that $X \cong \mathcal{M}_A$. 
\end{proposition}
\begin{proof}
Using the result \cite[Proposition 2.1]{pardini}, to any object of the stack $\mathcal{M}_A$, we associate a $\mathbb{Z}_N$-covering $\pi: C \to C'$, branched at the divisors $D_i$. 
The marked points $x_1, \ldots, x_n$ are chosen as the fibers of $p_1, \ldots, p_n$ under the covering map in such a way that $x_{\alpha(i)} \in D_i$. 

Proposition 2.1 of \cite{pardini} does not guarantee that the resulting covering space is connected, therefore we need to prove that $C$ is connected. If $n \geq 1$, then the covering $\pi$ must have a point of total ramification, which indeed implies that $C$ is connected. So let $n=0$. If $g'$ is equal to $1$, one can see that the only possibility if $g=2$ is that $N=2$ and $d_1=2$, thus there is a point of total ramification. If $g'$ is equal to $0$, the connectedness of $C$ is equivalent to the following numerical condition on the $(g,0)$-admissible datum:
$$
\gcd \left(N, \left\{i| \ d_i \neq 0\right\}\right)=1.
$$
This condition happens to be satisfied by all $(2,0)$-admissible data.

In the cases $n \geq 1$ or $g=2$, the stack $\mathcal{M}_A$ is connected due to Proposition \ref{connessione}; therefore its image under this correspondence is a connected component of $I(\mathcal{M}_{g,n})$.

Conversely, if $X$ is a connected component of $I_N(\mathcal{M}_{g,n})$, after quotienting by the group generated by the automorphism, the objects of $X$ are families of $\mathbb{Z}_N$-coverings. By invoking \cite[Proposition 2.1]{pardini} again in the opposite direction, we find the discrete datum $A$ and then the connected moduli space $\mathcal{M}_A$.
\end{proof}
\begin{remark} The map $\iota$ on the inertia stack, which we described in Remark \ref{mappaiota}, corresponds to the map of admissible data $(g',N,d_1, \ldots, d_{N-1},a_1, \ldots, a_{N-1}) \to (g',N,d_{N-1}, \ldots, d_1, \ldots, a_{N-1}, \ldots, a_1)$.
\end{remark}

There are seventeen $(2,0)$-admissible data, twenty-four data with $n=1$, twenty-six data with $n=2$, twenty-one with $n=3$, seven with $n=4$, and only one for $n=5,6$. We now list in \ref{tabellona} the $(2,0)$-admissible data, and we propose a name for each of them. Our names coincide with those given by Spencer \cite{spencer2} in the overlapping cases, we have only invented a new notation for $V.1$ and $V.2$.
We do not list the admissible data with $n>0$, since they can easily be determined starting from the ones with $n=0$. The complete list of them (and therefore, of the twisted sectors of $\mathcal{M}_{2,n}$) can be found in \cite[5.2]{paganitesi}.

\begin{notation} \label{notazionemg} In view of Proposition \ref{instack2}, if $A$ is a set of $2-$admissible data, it determines a twisted sector $X(A)$ of $\mathcal{M}_2$. Now if $\alpha: \{1,\ldots,n\} \to \mathbb{Z}_N^*$ is a function such that $|\alpha^{-1}(i)|=a_i$ (as in Proposition \ref{instack2}), we call $X(A)_{\alpha(1), \ldots, \alpha(n)}$ the corresponding twisted sector of $\mathcal{M}_{2,n}$. So for instance $III_{1122}$ and $III_{1212}$ are two distinct twisted sectors of the Inertia Stack $I(\mathcal{M}_{2,4})$. 
\end{notation}


\subsubsection{The compactification of the inertia stack of smooth genus $2$ curves}

Let $X$ be a twisted sector of $\mathcal{M}_{g,n}$. As we have already seen in Proposition \ref{instack2}, $X \cong \mathcal{M}_A$ for a certain $(g,n)$-admissible datum $A$ when $g$ equals $2$ or $n \geq 1$. In this section, we construct a compactification $\overline{\mathcal{M}}_A$. We follow the approach of the compactification of the moduli stack of stable maps to an orbifold, developed in \cite{abvis2} (see also \cite{agv2} and \cite{acv}). After that, we study the geometry of $\overline{\mathcal{M}}_A$, and in particular we study its cohomology and Chow groups. 

\begin{equation} \begin{tabular}{|c|c|c||c|c|}\hline \label{tabellona}
$g'$&$N$&$(d_1, \ldots, d_{N-1})$&Coarse Moduli Space& Name in \cite{spencer2}\\
\hline
&&&&\\
$0$&$2$&$(6)$& $\mathcal{M}_{0,6}/S_6$ & $(\tau)$ \\
$0$&$3$&$(2,2)$&  $\mathcal{M}_{0,4}/(S_2 \times S_2)$ & $(III)$ \\
$0$&$4$&$(1,2,1)$&  $\mathcal{M}_{0,4}/S_2$ & $(IV)$ \\
$0$&$5$&$(2,0,1,0)$ & $\mathcal{M}_{0,3}/S_2$ & $(X.4)$\\
$0$&$5$&$(0,1,0,2)$ &  $\mathcal{M}_{0,3}/S_2$ & $(X.6)$\\
$0$&$5$&$(1,2,0,0)$ & $\mathcal{M}_{0,3}/S_2$ & $(X.2)$\\
$0$&$5$&$(0,0,2,1)$ & $\mathcal{M}_{0,3}/S_2$ & $(X.8)$\\
$0$&$6$&$(2,0,0,1,0)$& $\mathcal{M}_{0,3}/S_2$ & $(V.1)$ \\
$0$&$6$&$(0,1,0,0,2)$&  $\mathcal{M}_{0,3}/S_2$ & $(V.2)$ \\
$0$&$6$&$(0,1,2,1,0)$&  $\mathcal{M}_{0,4}/S_2$ & $(VI)$ \\
$0$&$8$&$(1,0,1,1,0,0,0)$&$\mathcal{M}_{0,3}$ & $(VIII.1)$ \\
$0$&$8$&$(0,0,0,1,1,0,1)$&  $\mathcal{M}_{0,3}$ & $(VIII.2)$ \\
$0$&$10$&$(0,1,1,0,1,0,0,0,0)$&  $\mathcal{M}_{0,3}$ & $(X.7)$ \\
$0$&$10$&$(0,0,0,0,1,0,1,1,0)$&  $\mathcal{M}_{0,3}$ & $(X.3)$ \\
$0$&$10$&$(1,0,0,1,1,0,0,0,0)$&  $\mathcal{M}_{0,3}$ & $(X.1)$ \\
$0$&$10$&$(0,0,0,0,1,1,0,0,1)$&  $\mathcal{M}_{0,3}$ & $(X.9)$ \\
$1$&$2$&$(2)$& $4:1$ covering on $\mathcal{M}_{1,2}/S_2$ & $(II)$ \\
&&&&\\
\hline
\end{tabular}
\end{equation}

Let $A=(g',N, d_1, \ldots, d_{N-1}, a_1, \ldots, a_{N-1})$ be $(g,n)$-admissible datum, and $d:= \sum d_i$ and $n= \sum a_i$ as usual.

\begin{definition} \label{compactifiedtwisted} Let $\overline{\mathcal{M}}_{g',d}(B \mu_N)$ be the moduli stack of stable maps to $B \mu_N$ (see \cite{abvis2}, \cite{acv}). We define $\overline{\mathcal{M}}_A$ as the Zariski closure of $\mathcal{M}_A$ inside $[\overline{\mathcal{M}}_{g',d}(B \mu_N)/S_A]$, ($S_A$ defined in \ref{sa}).
\end{definition}

\noindent As $\mathcal{M}_A$ is connected by Proposition \ref{connessione}, it follows that $\overline{\mathcal{M}}_A$ is also connected.

\begin {definition} Let $X$ be a twisted sector of the inertia stack of $\overline{\mathcal{M}}_{g,n}$. We define $\partial X$ as the fiber product illustrated below:
$$
\xymatrix{\partial X \ar[r] \ar[d] \ar@{}|{\square}[dr]& X \ar[d]\\
\partial \overline{\mathcal{M}}_{g,n} \ar[r] & \overline{\mathcal{M}}_{g,n}.}
$$
\noindent We say that the twisted sector $X$ \emph{comes from the boundary} if $\partial X$ is equal to $X$ or, equivalently, if the image of $X$ under the canonical projection of the inertia stack is contained in the boundary of the moduli stack $\overline{\mathcal{M}}_{g,n}$, or, again equivalently, if its general element is not smooth. \label{bordo}
\end {definition}
\noindent If $X$ is a twisted sector of $I(\partial \overline{\mathcal{M}}_{g,n})$, we will need later in this paper a simple criterium to determine whether it comes from the boundary.
\begin{remark} \label{lisciabile} Let $X$ be a twisted sector of $I(\partial \overline{\mathcal{M}}_{g,n})$, whose general element is a couple $(C, \phi)$, where $C$ is a (family of) genus $g$, $n$-pointed, stable curves, and $\phi$ is an automorphism of $C$. Then $\phi$ induces a linear endomorphism $\phi'$ on the first order deformation vector space: $$\textrm{Ext}^1\left(\Omega_C( \sum x_i), \mathcal{O}_C \right)=H^1(C, \chom\left(\Omega_C(\sum x_i), \mathcal{O}_C)\right) \bigoplus H^0\left(C, \cext^1(\Omega_C(\sum x_i), \mathcal{O}_C) \right),$$ which respects the direct sum. Then the twisted sector $X$ comes from the boundary if and only if the invariant part of $ H^0\left(C, \cext^1(\Omega_C(\sum x_i), \mathcal{O}_C)\right)$ under the action of ${\phi'}$ is zero. If this is the case, we can also say that the twisted sector $X$ is non-smoothable, or that its general element is singular.
\end{remark}

Let us now go back to the compactification of the twisted sectors of $I(\mathcal{M}_{2,n})$. As a consequence of Proposition \ref{instack2}, we have:

\begin{corollary} \label{compsmooth} Let $A$ be a $(g,n)$-admissible datum, in the range $n\geq1$ or $g =2$. Then, any assignment $\alpha: \{1,\ldots,n\} \to \mathbb{Z}_N^*$ such that $|\alpha^{-1}(i)|=a_i$ induces an isomorphism of $\overline{\mathcal{M}}_A$ with a connected component of $I_N(\overline{\mathcal{M}}_{g,n})$. Conversely, to any connected component $X$ of $I_N(\overline{\mathcal{M}}_{g,n})$ \emph{that does not come from the boundary} (Definition \ref{bordo}), one can associate a $(g,n)$-admissible datum $A$, such that $X \cong \overline{\mathcal{M}}_A$.  
\end{corollary}

\begin{proof} The proof of the corollary follows by adapting the proof of Proposition \ref{instack2}. In the case when the covering curve $C$ turns out to be unstable, one applies the usual stabilization procedure.
\end{proof}

\noindent In other words, the $\overline{\mathcal{M}}_A$ of Definition \ref{compactifiedtwisted}, are all the twisted sectors that do not come from the boundary. 

\begin{notation} \label{notazionecompsmooth} Following Notation \ref{notazionemg}, we shall call the latters compactified twisted sectors $\overline{III}, \overline{II}_{11}, \ldots$. 
\end{notation}

We need to investigate the moduli stack $\overline{\mathcal{M}}_A$ a little more, in order to determine its cohomology groups. In particular, in the following, we will reduce the problem of computing the cohomology groups of the twisted sectors to the problem of computing the equivariant cohomology of $\overline{\mathcal{M}}_{0,n}$ under the action of the symmetric group $S_n$, which is then known (see \cite[5.8]{getzleroperads}). For this purpose, note that the twisted sectors of $\overline{\mathcal{M}}_{g,n}$ that do not come from the boundary split into two different classes: those whose general object is a covering of a genus $0$ curve, and those whose general object covers genus $g'$ curves with $g'>0$. The large majority of cases fall into the first class, and their cohomology is easily determined. On the contrary, the latter cases are fewer, but they require more work. 

\begin{theorem} (\emph{Cf.} \cite[p.2]{bayercadman}.) \label{bayercadman} Let $A$ be a $(g,n)$-admissible datum (see Definition \ref{2admissible}), with $g'$ equal to $0$. The space $\overline{\mathcal{M}}_A$ is then a $\mu_N$-gerbe over the quotient stack $[X / S_A]$, where $X$ is the stack constructed starting from $\overline{\mathcal{M}}_{0,\sum d_i}$ by successively applying the root construction (see \cite[Section 2]{bayercadman}).
\end {theorem}

\noindent The stack $X_{D,r}$, called the \emph{root of a line bundle with a section}, where $X$ is a scheme, $D$ is an effective Cartier divisor, and $r$ is a natural number, was introduced first in \cite{cadman} and \cite{agv2}. The only thing we shall need in the context of our computations is the following result:

\begin{proposition} (See \cite[Corollary 2.3.7]{cadman}.) Let $X$ be a scheme. If $X_{D,r}$ is obtained from $X$ by applying the root construction, the canonical map $X_{D,r} \to X$ exhibits $X$ as the coarse moduli space of $X_{D,r}$.
\end{proposition}

\begin{corollary} \label{corollariocadman} If $A$ is a $(g,n)$-admissible datum with $g'=0$, then the stack $\overline{\mathcal{M}}_A$ has the same rational Chow groups and rational cohomology groups as $\big[\overline{\mathcal{M}}_{0,\sum d_i}\big/ S_A \big]$.
\end{corollary}

There are only three twisted sectors in ${\mathcal{M}}_{2,n}$ whose general object is a covering of a genus $1$ curve. They are the space that we called $II$ in \ref{tabellona}, and the spaces obtained from it by marking the two points of total ramification of its general element. We call $II_1$ the space obtained by marking one point of total ramification and $II_{11}$ the space obtained by marking both the points of total ramification. They are moduli stacks of bielliptic curves, together with a choice of a bielliptic involution, and possibly marked points. The remainder of this section is devoted to investigate their geometry, in order to compute their rational cohomology and Chow groups. The following construction was discovered independently by Dan Petersen (see \cite[Section 2.3]{petersen}).

\begin{proposition} \label{duezero} The stack $II$ has the same coarse moduli space as $[\mathcal{M}_{0,5}/S_3]$. In the same way, $\overline{II}$  has the same coarse space as $[\overline{\mathcal{M}}_{0,5}/S_3]$.
\end{proposition}
\begin{proof}
We exhibit a morphism from $II$ to $[\mathcal{M}_{0,5}/S_3]$, which induces a bijection on objects. Let $C$ be a genus $2$ curve with an automorphism $\phi$ of order $2$ such that $E:= C/ \langle \phi \rangle$ is an elliptic curve and $C \to E$ is ramified at two points. Let $\pi_C: C \to \mathbb{P}^1$ be the double covering that induces the hyperelliptic structure on $C$. Since the hyperelliptic involution commutes with all automorphisms, there exists exactly one elliptic structure $\pi_E:E \to \mathbb{P}^1$, and a double covering $\mathbb{P}^1 \to \mathbb{P}^1$, branched in two points, such that the following diagram commutes:
$$
\xymatrix{C \ar[r] \ar[d]^{\pi_C} & E \ar[d]^{\pi_E}\\
\mathbb{P}^1 \ar[r] &\mathbb{P}^1.}
$$
If we call $0, 1, \infty, \lambda$ the branching points of $E$ on $\mathbb{P}^1$, and $p_1, p_2$ the branching points of the projection $C \to E$, then one between $\pi_E(p_1)$ and $\pi_E(p_2)$ must coincide with one among $0,1, \infty, \lambda$. Without loss of generality, we assume $\lambda= \pi_E(p_1)$. Therefore, we obtain the datum of $5$ points on $\mathbb{P}^1$: $0, 1, \infty, \lambda, q$. The map just defined from $II$ to $\mathcal{M}_{0,5}$ induces a map on $[\mathcal{M}_{0,5}/S_3]$ by composition with the quotient map. 

Let us now assume that $S_3$ acts on the first three points of $\mathcal{M}_{0,5}$. The inverse morphism from $[\mathcal{M}_{0,5}/S_3]$ to $II$ is obtained as follows. We can construct first a double covering $\gamma$ branched at the last two points and then a genus $2$ curve as a double covering whose branching points are the fibers of the three undistinguished marked points under $\gamma$. The genus $2$ curve thus constructed is bielliptic in two ways: an elliptic curve can be constructed by taking the double covering of the original genus $0$ curve, branched at the three undistinguished points and at one of the remaining $2$ points. The branch points of the bielliptic quotient are the two fibers in the elliptic curve of the remaining fifth point.

Finally, this construction extends to the boundary.
\end {proof}


\begin{corollary} \label{aggiunta2} The dimension of $A^1(\overline{II})=H^2(\overline{II})$ is three, and trivial otherwise. The cohomology of $II$ is one-dimensional in degrees $0,1,2$ and zero otherwise. \end{corollary}

We now deal with the two remaining twisted sectors we are interested in, which are $\overline{II}_1$ and $\overline{II}_{11}$. They are compactified moduli stacks of bielliptic curves, with a choice of a bielliptic involution and an ordering of the points of ramification.

\begin{proposition} \label{aggiunta} The stacks $\overline{II}_1$ and $\overline{II}_{11}$ are isomorphic, and the natural forgetful maps $\overline{II}_1 \to \overline{II}$ and $II_1 \to II$ induce isomorphisms in the rational Chow groups and in rational cohomology. 
\end{proposition}
\begin{proof} 
A preliminary step is to construct a map $II_1 \to \mathcal{M}_{1,2}$ that is an open embedding at the level of coarse moduli spaces, with complement $X_0(2)$: the locus where the second point is of two-torsion for the group structure on the genus $1$ having the first point as origin (see \cite[Chapter 3]{diamond} for $X_0(2)$ and other modular curves). In particular, $II_1$ is affine. We sketch the construction of this map. The space $II_1$ parametrizes smooth genus $1$ curves $C$, two distinct points $x_1,x_2$ on it, and a line bundle $L$ such that $2 L \equiv x_1+x_2$. Choosing $x_1$ as origin, $L$ corresponds to a point $x$ on $C$, such that $x \neq x_1, x \neq x_2$, and the three points satisfy the linear equivalence $2x \equiv x_1+x_2$. The point $x_2$ can be reconstructed from $x_1$ and $x$: the points $x_1$ and $x_2$ are distinct exactly when $2x \not \equiv 2 x_1$. 

 We start by proving the statement for the forgetful map $\overline{II}_1 \to \overline{II}$.  The first step is to show that $H^1(\overline{II}_1)=H^3(\overline{II}_{11})=0$ and $h^2(\overline{II}_1)=3$. The stack $II_1$ is isomorphic to the complement in $\mathcal{M}_{1,2}$ of the locus where the second point $x_2$ is of two-torsion. The space $X_0(2)$ is isomorphic to $\mathbb{P}^1$ minus $2$ points, while it is well known that $\mathcal{M}_{1,2}$ has trivial rational cohomology groups. From this, by the preliminary step, Poincar\'e duality, and the exact sequence of cohomology with compact support, we deduce that the cohomology of $II_1$ is one-dimensional in degrees $0,1$ and $2$. There are four irreducible divisors added to $II_1$ in its compactification $\overline{II}_1$ (Definition \ref{compactifiedtwisted}): by using again the exact sequence of compactly supported cohomology and the fact that $\overline{II}_1$ is a proper smooth stack, one gets the desired results on the Betti numbers of $\overline{II}_1$. 

The second step is to observe that $\overline{II}_1 \to \overline{II}$ describes the latter as the stack quotient of $\overline{II}_1$ by $S_2$ via the action that symmetrizes the two points of ramification. So the map $H^2(\overline{II}) \to H^2(\overline{II}_1)$ is the injection of the $S_2$-invariant part, and as the two vector spaces have the same dimension (by the first step and Corollary \ref{aggiunta2}), it is an isomorphism.

In the last step, we show that the map $\overline{II}_1 \to \overline{II}$ is an isomorphism also at the level of the Chow groups. Consider the two localization sequences for the inclusion of the boundary $\partial {II}$ in $\overline{II}$ (and resp. the same for $\overline{II}_1)$. By using the fact that $II_1$ and $II$ have trivial rational Chow groups, they become:
$$
\xymatrix{A^0 (\partial II) \ar[r]\ar[d]^{\partial t} &A^1 (\overline{II}) \ar[d]^{t} \ar[r]& 0 \\
A^0 (\partial II_1) \ar[r] &A^1 (\overline{II}_1) \ar[r]& 0.
}$$
The map $t$ is the inclusion of the $S_2$-invariant part. Since the map $\partial t$ is an isomorphism, the map $t$ must be an isomorphism too. As we already know that the cycle map $A^1(\overline{II}) \to H^2(\overline{II})$ (Corollary \ref{aggiunta2}), and that the map $H^2(\overline{II}) \to H^2(\overline{II}_1)$ (second step of this proof) are isomorphisms, we conclude that the cycle map $A^1(\overline{II}_1) \to H^2(\overline{II}_1)$ is also an isomorphism.

Let us now study the maps induced in cohomology and Chow groups by the forgetful map on the open parts: $II_1 \to II$. On the level of rational Chow groups, the map is trivial as both the two stacks have affine coarse moduli space. To see that the map induces an isomorphism in rational cohomology, by Poincar\'e duality, is the same as proving that it induces isomorphisms in compactly supported rational cohomology. This follows then by using the $5$-lemma on the two exact sequences of compactly supported cohomology involving respectively $II_1, \overline{II}_1, \partial II_1$ and $II, \overline{II}, \partial II$.
\end{proof}

\noindent The following result is then a consequence of Corollaries \ref{corollariocadman}, Corollary \ref{aggiunta2} and Proposition \ref{aggiunta}.

\begin{theorem} \label{generatodivisori} Let $X$ be a twisted sector of $\mathcal{M}_{2,k}$, and $\overline{X}$ its compactification according to Definition \ref{compactifiedtwisted}. Then the cycle map $A^*(\overline{X}) \to H^{2*}(\overline{X})$ is an isomorphism of graded vector spaces. Moreover, the Chow ring $A^*(\overline{X})$ is generated by the divisor classes.
\end{theorem}


\subsection{The inertia stack of moduli of curves with rational tails}
\label{rationaltails}
In this section, we reduce the study of the inertia stack of moduli of $n$-pointed curves with rational tails to the study of the inertia stack of moduli of smooth, $k$-pointed curves ($\forall k \leq n$). The latter were identified in the previous section. Then we exploit the same trick used in \cite{pagani1} of \virg{forgetting the rational tails} to study the inertia stack of the whole moduli stack of stable, $n$-pointed curves. 

\begin {definition} \label {wrt} If $(C,x_1, \ldots, x_n)$ is a stable curve, a \textit{rational tail} is a proper genus $0$ subcurve, which meets the closure of the complement at exactly $1$ point.  A stable curve will be said to be \textit{without rational tails} if it does not contain any rational tail. We will call the moduli stack of stable curves without rational tails $\overline{\mathcal{M}}_{g,n}^{NR}$. The moduli space of curves with rational tails $\mathcal{M}_{g,n}^{rt}$ is the moduli space of stable curves that have an irreducible smooth component of genus $g$ (and therefore the other components must be rational tails).
\end {definition} 

Let $I_1 \sqcup \ldots \sqcup I_k$ be a partition of $[n]$ made of non-empty subsets.  On the set of indices
 $$\mathcal{A}_{k,n}:= \{ (I_1, \ldots, I_k)| \ I_i \neq \emptyset, \ \sqcup I_i = [n]\},$$
 there is a natural action of the symmetric group: if $\sigma \in S_k$, $\sigma(I_1, \ldots, I_k):= (I_{\sigma(1)}, \ldots, I_{\sigma(k)})$. Then we let $\mathcal{A}_{k,n}/S_k$ be the quotient set, whose elements are a choice of a representative for every equivalence class. We define the map $j_g$:

\begin{equation} \label{disgiuntoprima}
j_g: \coprod_{k=1}^n \coprod_{(I_1, \ldots, I_k) \in \mathcal{A}_{k,n}/S_k} \mathcal{M}_{g,k} \times \overline{\mathcal{M}}_{0,I_1+1} \times \ldots \times \overline{\mathcal{M}}_{0,I_k+1} \to \mathcal{M}_{g,n}^{rt}.
\end{equation}
This map simply glues the genus $0$ curves (at the $+1$ points) together with the genus $g$ curves (at the points $1, \ldots, k$), see Fig. \ref{figureaddrat}. 

The map $j_g$ describes $\mathcal{M}_{g,n}^{rt}$ as a disjoint union of locally closed substacks. Moreover, $j_g$ induces an isomorphism on the automorphism groups of the objects. Thanks to Proposition \ref{strongrap}, we have: 
\begin{equation} \label{disgiunto}j_g^*\left(I(\mathcal{M}_{g,n}^{rt})\right) \cong \coprod_{k=1}^n \coprod_{(I_1, \ldots, I_k) \in \mathcal{A}_{k,n}/S_k} I(\mathcal{M}_{g,k}) \times \overline{\mathcal{M}}_{0,I_1+1} \times \ldots \times \overline{\mathcal{M}}_{0,I_k+1}.\end{equation}

We now compare $I(\mathcal{M}_{g,n}^{rt})$ and its pull-back $j_g^*(I(\mathcal{M}_{g,n}^{rt}))$ (see Definition \ref{pullinertia}). If $Y$ is a connected component of $I(\mathcal{M}_{g,n}^{rt})$, $j_g^*(Y)$ could potentially be a partition in locally closed substacks of $\mathcal{M}_{g,n}^{rt}$. But this is not the case if $Y$ is a twisted sector, by the following argument. Let $f:I(\mathcal{M}_{g,n}^{rt}) \to \mathcal{M}_{g,n}^{rt}$ be the canonical map from the inertia stack to the original Stack. Then $f(Y)$ is all contained in the image under $j_g$ of exactly one element in the disjoint union of \ref{disgiuntoprima}. This last argument does not apply when $Y$ is the untwisted sector, $I_1(\mathcal{M}_{g,n}^{rt})$. Indeed, $j_g^*(I_1(\mathcal{M}_{g,n}^{rt}))$ is a disjoint union of connected components, while $I_1(\mathcal{M}_{g,n}^{rt})$ is connected. We also observe that whenever $k>2g+2$, $I_{TW}(\mathcal{M}_{g,k})= \emptyset$. We have thus proved the following formula for the inertia stack of $\mathcal{M}_{g,n}^{rt}$.

 \begin{proposition} \label{twratcor} If $n>1$, the inertia stack of $\mathcal{M}_{g,n}^{rt}$ is isomorphic to:
 $$
I(\mathcal{M}_{g,n}^{rt})= (\mathcal{M}_{g,n}^{rt}, 1) \sqcup \coprod_{k=1}^{\min(n,2g+2)} I_{TW}(\mathcal{M}_{g,k}) \times \coprod_{(I_1, \ldots, I_k) \in \mathcal{A}_{k,n}/S_k}  \overline{\mathcal{M}}_{0,I_1+1} \times \ldots \times \overline{\mathcal{M}}_{0,I_k+1},
 $$
 where we make the convention that $\overline{\mathcal{M}}_{0,2}$ is a point.
 \end{proposition} 
 \begin{figure}[ht]
\centering
\psfrag{1}{$1$} 
\psfrag{2}{$2$}
\psfrag{3}{$3$}
\psfrag{4}{$4$}
\psfrag{A}{$I_1$} 
\psfrag{B}{$I_2$}
\psfrag{C}{$I_3$}
\psfrag{D}{$I_4$}
\psfrag{E}{$I_{TW}(\mathcal{M}_{g,4})$}
\includegraphics[scale=0.4]{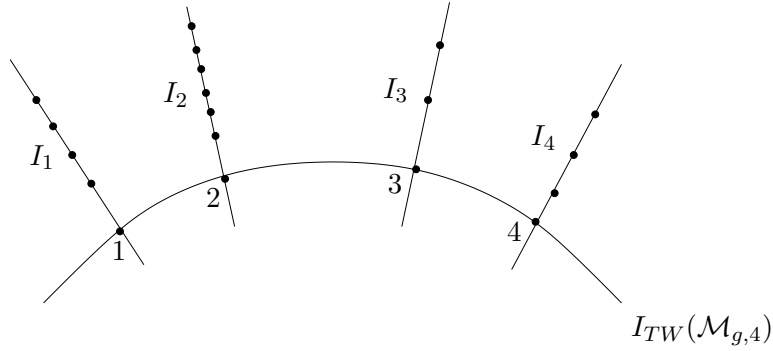}
\caption{Adding rational tails to $I_{TW}({\mathcal{M}}_{g,4})$ to obtain $I_{TW}(\mathcal{M}_{g,15}^{rt})$ with decomposition $[15]=I_1\sqcup I_2 \sqcup I_3 \sqcup I_4$ }
\label{figureaddrat}
\end{figure}


To complete this section, we make the analogous arguments for $\overline{\mathcal{M}}_{g,n}$. We then define the map $\overline{j}_g$ that glues the genus $0$ curves and the genus $g$ curves as above:

$$
\overline{j}_g: \coprod_{k=1}^n \coprod_{(I_1, \ldots, I_k) \in \mathcal{A}_{k,n}/S_k} \overline{\mathcal{M}}_{g,k}^{NR} \times \overline{\mathcal{M}}_{0,I_1+1} \times \ldots \times \overline{\mathcal{M}}_{0,I_k+1} \to \overline{\mathcal{M}}_{g,n}.
$$

\begin{remark} \label{lisciabile2} Here we want to point out that the process of gluing rational tails at marked points of twisted sectors produces new twisted sectors only for those marked points where the action of the automorphism is non-trivial. Let us be more precise. Let $X$ be a twisted sector of $\overline{\mathcal{M}}_{g,k}^{NR}$, whose general element is a couple $((C,x_1,\ldots, x_k), \phi)$, where $(C,x_1,\ldots,x_k)$ is a (family of) pointed curve(s), and $\phi$ is an automorphism of it. Then $\phi$ induces an action on $T_{x_l} C$, for $1 \leq l \leq k$. Suppose that $\bullet$ is a marked point on $C$ such that the induced action of $\phi$ on $T_{\bullet} C$ is trivial. Then gluing a genus $0$, marked curve at $\bullet$ produces a pointed curve $C'$ and an automorphism $\phi'$ on it, which is certainly not the general element of a twisted sector of $\overline{\mathcal{M}}_{g,n}$. Indeed, $\phi$ induces an action on $H^0\left(C', \cext^1(\Omega_C'(\sum x_i), \mathcal{O}_C')\right)$ whose invariant part is one-dimensional (cf. Remark \ref{lisciabile}), and therefore the corresponding node can be smoothed preserving the automorphism $\phi$. This produces a more general element of the same twisted sector.

An automorphism $\phi$ of a smooth curve $C$ with $g(C) \geq 2$ has only finitely many fixed points. Therefore, if $X$ is a twisted sector of $\overline{\mathcal{M}}_{g,k}^{NR}$ \emph{whose general element is a smooth curve}, the automorphism acts non-trivially on all the fixed points $x_1, \ldots, x_k$.
\end{remark} 

\noindent Let then $S(X)=S \subset [k]$ be the subset of the marked points where the action of $\phi$ is non-trivial. We define the generalization of $\mathcal{A}_{k,n}/S_k$:
$$
\mathcal{A}_{S,n}:= \{\{I_s\}_{s \in S}| \ \sqcup_{s \in S} I_s= [n], \ I_s \neq \emptyset \}.
$$
Let also $T_{g,k}$ be the set of twisted sectors of $I(\overline{\mathcal{M}}_{g,k}^{NR})$ (Definition \ref{wrt}). As a consequence of the analysis made in this section, we have:
\begin{proposition} \label{pirttheorem} If $n>1$, the inertia stack of $\overline{\mathcal{M}}_{g,n}$ is isomorphic to:
$$
I(\overline{\mathcal{M}}_{g,n})= \left( \overline{\mathcal{M}}_{g,n}, 1 \right)\sqcup \coprod_{k=1}^{n} \coprod_{X \in T_{g,k}} X \times \coprod_{\{I_s\} \in \mathcal{A}_{S(X),n}} \prod_{s \in S(X)} \overline{\mathcal{M}}_{0,I_s+1} .
$$ 
 \end{proposition}
 \begin{notation} \label{notazionemgnrt} Following Notation \ref{notazionemg} and \ref{notazionecompsmooth}, if $X_{\alpha(1), \ldots, \alpha(k)}$ is a twisted sector of $\mathcal{M}_{2,k}$, and $I_1, \ldots, I_k$ is a partition of $[n]$ in non-empty subsets, we shall call $X_{\alpha(1), \ldots, \alpha(k)}^{I_1, \ldots, I_k}$ the corresponding twisted sector of $\mathcal{M}_{2,n}^{rt}$ obtained by gluing the rational tails with marked points in $I_1, \ldots, I_k$. Similarly if  $\overline{X}_{\alpha(1), \ldots, \alpha(k)}$ is a twisted sector of $\overline{\mathcal{M}}_{2,k}$.
 \end{notation}

 While we have a description of the inertia stack $I(\mathcal{M}_{2,k})$ and its compactification from the previous sections, we have not studied the whole inertia stack $I(\overline{\mathcal{M}}_{2,k}^{NR})$ yet. We will complete the study of the latter in the following section.
 
\subsection{The inertia stack of moduli of stable curves of genus $2$}

\label{dalbordo}

In this section, we complete the description of the inertia stack of $\overline{\mathcal{M}}_{2,n}$. 
Using Proposition \ref{pirttheorem}, to describe $I(\overline{\mathcal{M}}_{2,n})$ it is enough to provide all the twisted sectors of $\overline{\mathcal{M}}_{2,k}^{NR}$ (Definition \ref{wrt}), and to show which marked points among the $k$ are suitable for attaching rational tails \emph{i.e.} to describe which are the marked points where the automorphism under consideration acts non-trivially (\emph{cf.} Remark \ref{lisciabile2}). As we have already described the twisted sectors of $I(\overline{\mathcal{M}}_{2,k}^{NR})$ whose general element is smooth, we have to study the twisted sectors of $I(\partial \overline{\mathcal{M}}_{2,k}^{NR})$ that come from the boundary, according to Definition \ref{bordo} (\emph{cf.} Remark \ref{lisciabile}).

We start with the case $k=0$. Let us consider the following two stable genus $2$ unmarked graphs:

\begin{equation} \label{grafigenere2}
    \begin{tikzpicture}[baseline]
      \path(0,0) ellipse (2 and 1);
      \node (A0) at (0:1) {$\scriptstyle{{\hspace{0.08cm}}_1^{\hspace{0.2cm} }}$};

      \draw (A0) .. controls +(-15:1.2) and +(15:1.2) .. (A0);
    \end{tikzpicture}
    \begin{tikzpicture}[baseline]
      \path(0,0) ellipse (2 and 1);
      \node (A0) at (0:1) {$\scriptstyle{{\hspace{0.08cm}}_1^{\hspace{0.2cm} }}$};
      \node (A1) at (180:1) {$\scriptstyle{{\hspace{0.08cm}}_1^{\hspace{0.2cm} }}$};

      \path (A0) edge [bend left=0] (A1);
    \end{tikzpicture}
\end {equation}

\noindent These two graphs correspond to two divisors in $\overline{\mathcal{M}}_{2}$. We call the two graphs respectively $\Gamma_0$ and $\Gamma_1$, and the closed substacks they correspond to are usually referred to as $\Delta_0$ and $\Delta_1$ (see \cite[Part III]{mumford}). 

We begin by studying the case when $C$ is a curve whose corresponding dual graph is $\Gamma_1$. Then suppose that the automorphism $\phi$ does not exchange the two irreducible components of the curve. The twisted sectors of $\partial \overline{\mathcal{M}}_2$ that correspond to this case are in bijection with the twisted sectors of $\overline{\mathcal{M}}_{1,1}\times \overline{\mathcal{M}}_{1,1}$, or, in other words, to couples of connected components $(X_1, \phi_1), (X_2, \phi_2)$ of $I($\mb{1}{1}$)$ (see \cite[Section 3]{pagani1} for $I$(\mb{1}{1}$)$). To find out if this is a new twisted sector of $\overline{\mathcal{M}}_2$ (\emph{i.e.}, if it comes from the boundary, see Definition \ref{bordo}), it is enough to check if the resulting node is smoothable (\emph{cf.} Remark \ref{lisciabile}). 

One can check that there are $31$ twisted sectors of $\overline{\mathcal{M}}_2$ that correspond to the latter description. One is two-dimensional and its moduli space is $\mathbb{P}^1 \times \mathbb{P}^1$, $12$ have moduli space isomorphic to $\mathbb{P}^1$, and then there are $18$ stacky points (see the first set of figures of \cite[Construction 5.25]{paganitesi} for more details).

Next, if $C$ is a curve corresponding to the graph $\Gamma_1$, we consider the case when the automorphism $\phi$ exchanges the two irreducible components. In this case, the twisted sectors correspond simply to the twisted sectors of \mb{1}{1} ($1$ has moduli space isomorphic to $\mathbb{P}^1$, and then there are $6$ stacky points). See the last two sets of figures of 
\cite[Construction 5.25]{paganitesi}.

Finally, we deal with those twisted sectors of $\overline{\mathcal{M}}_2$ whose general element is a curve whose dual graph is $\Gamma_0$. Along the same lines, one can see that there are $8$ twisted sectors, and that they are all stacky points. We refer again to \cite[Construction 5.27]{paganitesi} for more details.

We now study the twisted sectors of $\overline{\mathcal{M}}_{2,k}^{NR}$ that are contained in the boundary ($k \geq 1$). If $X$ is such a twisted sector, its general element is a marked curve with a distinguished automorphism $(C,x_1, \ldots, x_k, \phi)$. The resulting couple $(\tilde{C}, \tilde{\phi})$, obtained by forgetting the marked points and then stabilizing, can be of four different types (as we have just seen in the paragraphs above):
\begin{enumerate}
\item $\tilde{C}$ has dual graph $\Gamma_1$ and $\tilde{\phi}$ acts fixing the two components of genus $1$;
\item $\tilde{C}$ has dual graph $\Gamma_1$ and $\tilde{\phi}$ acts exchanging the two components of genus $1$;
\item $\tilde{C}$ has dual graph $\Gamma_0$ and $\tilde{\phi}$ acts fixing the two branches of the node;
\item $\tilde{C}$ has dual graph $\Gamma_0$ and $\tilde{\phi}$ acts exchanging the two branches of the node.
\end{enumerate}

We are now going to list in detail all the twisted sectors of $\partial \overline{\mathcal{M}}_{2,k}^{NR}$ that are non-smoothable (Definition \ref{bordo}, Remark \ref{lisciabile}), dividing them according to the four cases above. The marked points $p$ such that the induced action of the automorphism is non-trivial on the tangent space\footnote{And therefore, those suitable for gluing rational tails in such a way that the resulting operation is non-smoothable, see Proposition \ref{pirttheorem} and Remark \ref{lisciabile2}.} to the curve in $p$ are displayed with a dot $\bullet$ in their end points. 

We will use the notation $T_i, T_i', \tilde{T}_i, \tilde{T}_i', T^{\rho}_i$ for certain subsets of the sets of compactified twisted sectors of quotients of the kind $[\mathcal{M}_{1,i}/S]$ that are briefly introduced in the next few lines, and thoroughly discussed in \ref{appendicea}. We will plug these compactified twisted sectors in stable $n$-pointed graphs of genus $2$ (these graphs and their pictures come from \cite{mp}). 

\begin{example} Case $1$. The automorphism of every twisted sector induces the identity automorphism on the corresponding dual graph: \label{case1}

\begin{center}
\begin{tabular}{c@{}cc@{}cc@{}cc@{}c}
\begin{tikzpicture}[baseline]
      \path(0,0) ellipse (2 and 1);
      \tikzstyle{level 1}=[counterclockwise from=-60,level distance=9mm,sibling angle=120]
      \node (A0) at (0:1) {$\scriptstyle{T_1}$};
      \tikzstyle{level 1}=[counterclockwise from=120,level distance=9mm,sibling angle=20]
      \node (A1) at (180:1) {$\scriptstyle{1_n}$} child child child child child child child;

      \path (A0) edge [bend left=0] (A1);
    \end{tikzpicture}

\begin{tikzpicture}[baseline]
      \path(0,0) ellipse (2 and 1);
      \tikzstyle{level 1}=[counterclockwise from=0,level distance=9mm,sibling angle=120]
      \node (A0) at (0:1) {$\scriptstyle{T_2}$} child{[fill] circle (2pt)};
      \tikzstyle{level 1}=[counterclockwise from=120,level distance=9mm,sibling angle=20]
      \node (A1) at (180:1) {$\scriptstyle{1_n}$} child child child child child child child;

      \path (A0) edge [bend left=0] (A1);
    \end{tikzpicture}

\begin{tikzpicture}[baseline]
      \path(0,0) ellipse (2 and 1);
      \tikzstyle{level 1}=[counterclockwise from=-60,level distance=9mm,sibling angle=120]
      \node (A0) at (0:1) {$\scriptstyle{T_3}$} child{[fill] circle (2pt)} child{[fill] circle (2pt)};
      \tikzstyle{level 1}=[counterclockwise from=120,level distance=9mm,sibling angle=20]
      \node (A1) at (180:1) {$\scriptstyle{1_n}$} child child child child child child child;

      \path (A0) edge [bend left=0] (A1);
    \end{tikzpicture}
      \begin{tikzpicture}[baseline]
      \path(0,0) ellipse (2 and 1);
      \tikzstyle{level 1}=[counterclockwise from=-60,level distance=9mm,sibling angle=60]
      \node (A0) at (0:1) {$\scriptstyle{T_4}$} child{[fill] circle (2pt)} child{[fill] circle (2pt)} child{[fill] circle (2pt)};
      \tikzstyle{level 1}=[counterclockwise from=120,level distance=9mm,sibling angle=20]
      \node (A1) at (180:1) {$\scriptstyle{1_{n}}$} child child child child child child child;

      \path (A0) edge [bend left=0] (A1);
    \end{tikzpicture}
\\

 \begin{tikzpicture}[baseline]
      \path(0,0) ellipse (2 and 2);
      \node (A0) at (0:1) {$\scriptstyle{T_1}$};
      \node (A1) at (240:1) {$\scriptstyle{T_1}$};
      \tikzstyle{level 1}=[counterclockwise from=75,level distance=9mm,sibling angle=15]
\node (A2) at (120:1) {$\scriptstyle{0_{n}}$} child child child child child child;
      
      \path (A0) edge [bend left=0] (A2);
      \path (A1) edge [bend left=0] (A2);
    \end{tikzpicture}

 \begin{tikzpicture}[baseline]
      \path(0,0) ellipse (2 and 2);
      \tikzstyle{level 1}=[counterclockwise from=-30,level distance=9mm,sibling angle=15]
      \node (A0) at (0:1) {$\scriptstyle{T_2}$} child{[fill] circle (2pt)};
      \node (A1) at (240:1) {$\scriptstyle{T_1}$};
      \tikzstyle{level 1}=[counterclockwise from=75,level distance=9mm,sibling angle=15]
\node (A2) at (120:1) {$\scriptstyle{0_{n}}$} child child child child child child;
      
      \path (A0) edge [bend left=0] (A2);
      \path (A1) edge [bend left=0] (A2);
    \end{tikzpicture}

 \begin{tikzpicture}[baseline]
      \path(0,0) ellipse (2 and 2);
      
      \tikzstyle{level 1}=[counterclockwise from=-45,level distance=9mm,sibling angle=30]
      \node (A0) at (0:1) {$\scriptstyle{T_3}$} child{[fill] circle (2pt)} child{[fill] circle (2pt)};
      \node (A1) at (240:1) {$\scriptstyle{T_1}$};
      \tikzstyle{level 1}=[counterclockwise from=75,level distance=9mm,sibling angle=15]
\node (A2) at (120:1) {$\scriptstyle{0_{n}}$} child child child child child child;
      
      \path (A0) edge [bend left=0] (A2);
      \path (A1) edge [bend left=0] (A2);
    \end{tikzpicture}

 \begin{tikzpicture}[baseline]
      \path(0,0) ellipse (2 and 2);
      \tikzstyle{level 1}=[counterclockwise from=-60,level distance=9mm,sibling angle=30]
      \node (A0) at (0:1) {$\scriptstyle{T_4}$}child{[fill] circle (2pt)} child{[fill] circle (2pt)} child{[fill] circle (2pt)};
      \node (A1) at (240:1) {$\scriptstyle{T_1}$};
      \tikzstyle{level 1}=[counterclockwise from=75,level distance=9mm,sibling angle=15]
\node (A2) at (120:1) {$\scriptstyle{0_{n}}$} child child child child child child;
      
      \path (A0) edge [bend left=0] (A2);
      \path (A1) edge [bend left=0] (A2);
    \end{tikzpicture}

 \begin{tikzpicture}[baseline]
      \path(0,0) ellipse (2 and 2);
      
      \tikzstyle{level 1}=[counterclockwise from=-30,level distance=9mm,sibling angle=15]
      \node (A0) at (0:1) {$\scriptstyle{T_2}$}child{[fill] circle (2pt)};
      \tikzstyle{level 1}=[counterclockwise from=-90,level distance=9mm,sibling angle=30]
      \node (A1) at (240:1) {$\scriptstyle{T_2}$}child{[fill] circle (2pt)};
      \tikzstyle{level 1}=[counterclockwise from=75,level distance=9mm,sibling angle=15]
\node (A2) at (120:1) {$\scriptstyle{0_{n}}$} child child child child child child;
      
      \path (A0) edge [bend left=0] (A2);
      \path (A1) edge [bend left=0] (A2);
    \end{tikzpicture}
    \\ 
 \begin{tikzpicture}[baseline]
      \path(0,0) ellipse (2 and 2);
      \tikzstyle{level 1}=[counterclockwise from=-45,level distance=9mm,sibling angle=30]
      \node (A0) at (0:1) {$\scriptstyle{T_3}$}child{[fill] circle (2pt)} child{[fill] circle (2pt)};
      
      \tikzstyle{level 1}=[counterclockwise from=-90,level distance=9mm,sibling angle=30]
      \node (A1) at (240:1) {$\scriptstyle{T_2}$}child{[fill] circle (2pt)};
      \tikzstyle{level 1}=[counterclockwise from=75,level distance=9mm,sibling angle=15]
\node (A2) at (120:1) {$\scriptstyle{0_{n}}$} child child child child child child;
      
      \path (A0) edge [bend left=0] (A2);
      \path (A1) edge [bend left=0] (A2);
    \end{tikzpicture}

 \begin{tikzpicture}[baseline]
      \path(0,0) ellipse (2 and 2);
      \tikzstyle{level 1}=[counterclockwise from=-60,level distance=9mm,sibling angle=30]
      \node (A0) at (0:1) {$\scriptstyle{T_4}$}child{[fill] circle (2pt)} child{[fill] circle (2pt)} child{[fill] circle (2pt)};
      
      \tikzstyle{level 1}=[counterclockwise from=-90,level distance=9mm,sibling angle=30]
      \node (A1) at (240:1) {$\scriptstyle{T_2}$}child{[fill] circle (2pt)};
      \tikzstyle{level 1}=[counterclockwise from=75,level distance=9mm,sibling angle=15]
\node (A2) at (120:1) {$\scriptstyle{0_{n}}$} child child child child child child;
      
      \path (A0) edge [bend left=0] (A2);
      \path (A1) edge [bend left=0] (A2);
    \end{tikzpicture}

 \begin{tikzpicture}[baseline]
      \path(0,0) ellipse (2 and 2);
      \tikzstyle{level 1}=[counterclockwise from=-45,level distance=9mm,sibling angle=30]
      \node (A0) at (0:1) {$\scriptstyle{T_3}$}child{[fill] circle (2pt)} child{[fill] circle (2pt)};
      
      \tikzstyle{level 1}=[counterclockwise from=-105,level distance=9mm,sibling angle=30]
      \node (A1) at (240:1) {$\scriptstyle{T_3}$}child{[fill] circle (2pt)} child{[fill] circle (2pt)};
      \tikzstyle{level 1}=[counterclockwise from=75,level distance=9mm,sibling angle=15]
\node (A2) at (120:1) {$\scriptstyle{0_{n}}$} child child child child child child;
      
      \path (A0) edge [bend left=0] (A2);
      \path (A1) edge [bend left=0] (A2);
    \end{tikzpicture}

 \begin{tikzpicture}[baseline]
      \path(0,0) ellipse (2 and 2);
      \tikzstyle{level 1}=[counterclockwise from=-60,level distance=9mm,sibling angle=30]
      \node (A0) at (0:1) {$\scriptstyle{T_4}$}child{[fill] circle (2pt)} child{[fill] circle (2pt)} child{[fill] circle (2pt)};
      
      \tikzstyle{level 1}=[counterclockwise from=-105,level distance=9mm,sibling angle=30]
      \node (A1) at (240:1) {$\scriptstyle{T_3}$}child{[fill] circle (2pt)} child{[fill] circle (2pt)};
      \tikzstyle{level 1}=[counterclockwise from=75,level distance=9mm,sibling angle=15]
\node (A2) at (120:1) {$\scriptstyle{0_{n}}$} child child child child child child;
      
      \path (A0) edge [bend left=0] (A2);
      \path (A1) edge [bend left=0] (A2);
    \end{tikzpicture}

 \begin{tikzpicture}[baseline]
      \path(0,0) ellipse (2 and 2);
      \tikzstyle{level 1}=[counterclockwise from=-60,level distance=9mm,sibling angle=30]
      \node (A0) at (0:1) {$\scriptstyle{T_4}$}child{[fill] circle (2pt)} child{[fill] circle (2pt)} child{[fill] circle (2pt)};
      
      \tikzstyle{level 1}=[counterclockwise from=-120,level distance=9mm,sibling angle=30]
      \node (A1) at (240:1) {$\scriptstyle{T_4}$}child{[fill] circle (2pt)} child{[fill] circle (2pt)} child{[fill] circle (2pt)};
      \tikzstyle{level 1}=[counterclockwise from=75,level distance=9mm,sibling angle=15]
\node (A2) at (120:1) {$\scriptstyle{0_{n}}$} child child child child child child;
      
      \path (A0) edge [bend left=0] (A2);
      \path (A1) edge [bend left=0] (A2);
    \end{tikzpicture}
    
    \end {tabular}\end{center}

     \noindent Here $T_i$ is the set of compactified twisted sectors of $\mathcal{M}_{1,i}$ (a list of those is in \cite[Section 3]{pagani1}). Note that all these sectors are automatically non-smoothable (Remark \ref{lisciabile}) if the number of marked points $n$ on the genus $0$ component is greater than $0$. When this number is zero, the genus $0$ component contracts, and some of the elements in the list might be smoothable. These cases are therefore quite special and must be described separately. See the first set of figures in \cite[Construction 5.25]{paganitesi}. 
     
    \end{example}
    
    \begin{example} \label{case2} Case $2$. The automorphism induced on the stable graph associated to every curve exchanges the two components of genus $1$:
 
 \begin{center}   
\begin{tabular}{c@{}cc@{}cc@{}cc@{}c}

\begin{tikzpicture}[baseline]
      \path(0,0) ellipse (2 and 1);
      \tikzstyle{level 1}=[counterclockwise from=-60,level distance=9mm,sibling angle=120]
      \node (A0) at (0:1) {$\scriptstyle{T_1}$};
      \tikzstyle{level 1}=[counterclockwise from=120,level distance=9mm,sibling angle=20]
      \node (A1) at (180:1) {$\scriptstyle{T_1}$};

      \path (A0) edge [bend left=0] (A1);
    \end{tikzpicture}
        
    \begin{tikzpicture}[baseline]
      \path(0,0) ellipse (2 and 2);
      \node (A0) at (0:1) {$\scriptstyle{T_1'}$};
      \node (A1) at (240:1) {$\scriptstyle{ T_1' }$};
      \tikzstyle{level 1}=[counterclockwise from=120,level distance=9mm,sibling angle=60]
\node (A2) at (120:1) {$\scriptstyle{0_{1}^*}$} child{[fill] circle (2pt)};
      
      \path (A0) edge [bend left=0] (A2);
      \path (A1) edge [bend left=0] (A2);
    \end{tikzpicture}
    
    \begin{tikzpicture}[baseline]
      \path(0,0) ellipse (2 and 2);
      \node (A0) at (0:1) {$\scriptstyle{T_1'}$};
      \node (A1) at (240:1) {$\scriptstyle{T_1' }$};
      \tikzstyle{level 1}=[counterclockwise from=90,level distance=9mm,sibling angle=60]
\node (A2) at (120:1) {$\scriptstyle{0_{2}^*}$} child{[fill] circle (2pt)} child{[fill] circle (2pt)};
      
      \path (A0) edge [bend left=0] (A2);
      \path (A1) edge [bend left=0] (A2);
    \end{tikzpicture}

\end{tabular} \end{center}

  \noindent The twisted sectors in the same graph must be the same, in order to preserve the automorphism $\rho$. Here $T_1'$ is the set of compactified twisted sectors of $\mathcal{M}_{1,1}$ (see again \cite[Section 3]{pagani1}), whose corresponding automorphism is not $-1$ (because having the involutive twisted sector repeated twice produces a smoothable twisted sector).   
  The vertices $0_1^*$ and $0_2^*$ correspond to the twisted sectors of $\big[\overline{\mathcal{M}}_{0,3}\big/ S_2 \big]$ and $\big[\overline{\mathcal{M}}_{0,4}\big/ S_2 \big]$ (see Definition \ref{inerziazero}).    
\end{example}

\begin{example} \label{case3} Case $3$. The automorphism induced on the graph by the automorphism of the twisted sector is the identity:

 \begin{center}\begin{tabular}{c@{}c|c@{}c|c@{}c|c@{}c}
 
 \begin{tikzpicture}[baseline]
      \path(0,0) ellipse (2 and 1);
      \tikzstyle{level 1}=[counterclockwise from=180,level distance=9mm,sibling angle=60]
      
      \node (A0) at (0:1) {$\scriptstyle{\tilde{T_2'}}$};

      \draw (A0) .. controls +(-15:1.2) and +(15:1.2) .. (A0);
    \end{tikzpicture}
     
 \begin{tikzpicture}[baseline]
      \path(0,0) ellipse (2 and 1);
      \tikzstyle{level 1}=[counterclockwise from=180,level distance=9mm,sibling angle=60]
      
      \node (A0) at (0:1) {$\scriptstyle{\tilde{T_3'}}$} child{[fill] circle (2pt)};

      \draw (A0) .. controls +(-15:1.2) and +(15:1.2) .. (A0);
    \end{tikzpicture}
    
 \begin{tikzpicture}[baseline]
      \path(0,0) ellipse (2 and 1);
      \tikzstyle{level 1}=[counterclockwise from=150,level distance=9mm,sibling angle=60]
      
      \node (A0) at (0:1) {$\scriptstyle{\tilde{T_4'}}$}child{[fill] circle (2pt)} child{[fill] circle (2pt)};

      \draw (A0) .. controls +(-15:1.2) and +(15:1.2) .. (A0);
    \end{tikzpicture}
    \\   
     \begin{tikzpicture}[baseline]
      \path(0,0) ellipse (2 and 1);
      \node (A0) at (0:1) {$\scriptstyle{\tilde{T_2}}$};
      \tikzstyle{level 1}=[counterclockwise from=120,level distance=9mm,sibling angle=20]
      \node (A1) at (180:1) {$\scriptstyle{0_n}$} child child child child child child child;
      \path (A0) edge [bend left=-15] (A1);
      \path (A0) edge [bend left=15] (A1);
    \end{tikzpicture}
    
     \begin{tikzpicture}[baseline]
      \path(0,0) ellipse (2 and 1);
      \tikzstyle{level 1}=[counterclockwise from=0,level distance=9mm,sibling angle=20]
      
      \node (A0) at (0:1) {$\scriptstyle{\tilde{T_3}}$} child{[fill] circle (2pt)};
      \tikzstyle{level 1}=[counterclockwise from=120,level distance=9mm,sibling angle=20]
      \node (A1) at (180:1) {$\scriptstyle{0_n}$} child child child child child child child;
      \path (A0) edge [bend left=-15] (A1);
      \path (A0) edge [bend left=15] (A1);
    \end{tikzpicture}
    
     \begin{tikzpicture}[baseline]
      \path(0,0) ellipse (2 and 1);
      \tikzstyle{level 1}=[counterclockwise from=-15,level distance=9mm,sibling angle=30]
      
      \node (A0) at (0:1) {$\scriptstyle{\tilde{T_4}}$} child{[fill] circle (2pt)} child{[fill] circle (2pt)};
      \tikzstyle{level 1}=[counterclockwise from=120,level distance=9mm,sibling angle=20]
      \node (A1) at (180:1) {$\scriptstyle{0_n}$} child child child child child child child;
      \path (A0) edge [bend left=-15] (A1);
      \path (A0) edge [bend left=15] (A1);
    \end{tikzpicture}
 \end{tabular} \end{center}

    \noindent Here $\tilde{T_i}$ is the set of compactified twisted sectors (Definition \ref{compactifiedtwisted}) of $[\mathcal{M}_{1,i}/S_2]$ such that the distinguished automorphism fixes the two marked points symmetrized by the $S_2$ action. The set $\tilde{T_i'}$ is the set of twisted sectors in $\tilde{T_i}$ where the automorphism is not an involution. See \ref{appendicea}, and especially its last paragraph.
 \end{example}

\begin{example} \label{case4} Case $4$. The automorphism induced on the associated stable graph by the automorphism of the twisted sector exchanges the two edges (or the two branches of the same edge):

\begin{center}
 \begin{tabular}{c@{}c|c@{}c|c@{}c|c@{}c}

 \begin{tikzpicture}[baseline]
      \path(0,0) ellipse (2 and 1);
      \tikzstyle{level 1}=[counterclockwise from=180,level distance=9mm,sibling angle=60]
      
      \node (A0) at (0:1) {$\scriptstyle{{T_2^{\rho}}}$};

      \draw (A0) .. controls +(-15:1.2) and +(15:1.2) .. (A0);
    \end{tikzpicture}
    
 \begin{tikzpicture}[baseline]
      \path(0,0) ellipse (2 and 1);
      \tikzstyle{level 1}=[counterclockwise from=180,level distance=9mm,sibling angle=60]
      
      \node (A0) at (0:1) {$\scriptstyle{{T_3^{\rho}}}$} child{[fill] circle (2pt)};

      \draw (A0) .. controls +(-15:1.2) and +(15:1.2) .. (A0);
    \end{tikzpicture}
    
 \begin{tikzpicture}[baseline]
      \path(0,0) ellipse (2 and 1);
      \tikzstyle{level 1}=[counterclockwise from=150,level distance=9mm,sibling angle=60]
      
      \node (A0) at (0:1) {$\scriptstyle{{T_4^{\rho}}}$}child{[fill] circle (2pt)} child{[fill] circle (2pt)};

      \draw (A0) .. controls +(-15:1.2) and +(15:1.2) .. (A0);
    \end{tikzpicture}
    \\   
     \begin{tikzpicture}[baseline]
      \path(0,0) ellipse (2 and 1);
      \node (A0) at (0:1) {$\scriptstyle{{T_2^{\rho}}}$};
      \tikzstyle{level 1}=[counterclockwise from=180,level distance=9mm,sibling angle=60]
      \node (A1) at (180:1) {$\scriptstyle{0_1^*}$} child{[fill] circle (2pt)};
      \path (A0) edge [bend left=-15] (A1);
      \path (A0) edge [bend left=15] (A1);
    \end{tikzpicture}
    \begin{tikzpicture}[baseline]
      \path(0,0) ellipse (2 and 1);
      \tikzstyle{level 1}=[counterclockwise from=0,level distance=9mm,sibling angle=20]
      
      \node (A0) at (0:1) {$\scriptstyle{{T_3^{\rho}}}$} child{[fill] circle (2pt)};
      \tikzstyle{level 1}=[counterclockwise from=180,level distance=9mm,sibling angle=60]
      \node (A1) at (180:1) {$\scriptstyle{0_1^*}$} child{[fill] circle (2pt)};
      \path (A0) edge [bend left=-15] (A1);
      \path (A0) edge [bend left=15] (A1);
    \end{tikzpicture}

     \begin{tikzpicture}[baseline]
      \path(0,0) ellipse (2 and 1);
      \tikzstyle{level 1}=[counterclockwise from=-15,level distance=9mm,sibling angle=30]
      
      \node (A0) at (0:1) {$\scriptstyle{{T_4^{\rho}}}$} child{[fill] circle (2pt)} child{[fill] circle (2pt)};
      \tikzstyle{level 1}=[counterclockwise from=180,level distance=9mm,sibling angle=60]
      \node (A1) at (180:1) {$\scriptstyle{0_1^*}$} child{[fill] circle (2pt)};
      \path (A0) edge [bend left=-15] (A1);
      \path (A0) edge [bend left=15] (A1);
    \end{tikzpicture}
     \begin{tikzpicture}[baseline]
      \path(0,0) ellipse (2 and 1);
      \node (A0) at (0:1) {$\scriptstyle{{T_2^{\rho}}}$};
      \tikzstyle{level 1}=[counterclockwise from=150,level distance=9mm,sibling angle=60]
      \node (A1) at (180:1) {$\scriptstyle{0_2^*}$} child{[fill] circle (2pt)} child{[fill] circle (2pt)};
      \path (A0) edge [bend left=-15] (A1);
      \path (A0) edge [bend left=15] (A1);
    \end{tikzpicture}
    \begin{tikzpicture}[baseline]
      \path(0,0) ellipse (2 and 1);
      \tikzstyle{level 1}=[counterclockwise from=0,level distance=9mm,sibling angle=20]
      
      \node (A0) at (0:1) {$\scriptstyle{{T_3^{\rho}}}$} child{[fill] circle (2pt)};
      \tikzstyle{level 1}=[counterclockwise from=150,level distance=9mm,sibling angle=60]
      \node (A1) at (180:1) {$\scriptstyle{0_2^*}$} child{[fill] circle (2pt)} child{[fill] circle (2pt)};
      \path (A0) edge [bend left=-15] (A1);
      \path (A0) edge [bend left=15] (A1);
    \end{tikzpicture}

     \begin{tikzpicture}[baseline]
      \path(0,0) ellipse (2 and 1);
      \tikzstyle{level 1}=[counterclockwise from=-15,level distance=9mm,sibling angle=30]
      
      \node (A0) at (0:1) {$\scriptstyle{{T_4^{\rho}}}$} child{[fill] circle (2pt)} child{[fill] circle (2pt)};
      \tikzstyle{level 1}=[counterclockwise from=150,level distance=9mm,sibling angle=60]
      \node (A1) at (180:1) {$\scriptstyle{0_2^*}$} child{[fill] circle (2pt)} child{[fill] circle (2pt)};
      \path (A0) edge [bend left=-15] (A1);
      \path (A0) edge [bend left=15] (A1);
    \end{tikzpicture}
\end{tabular} \end{center}

\noindent See Definition \ref{inerziazero} for the notation on the twisted sectors $0_1^*$ and $0_2^*$. The set ${T_i^{\rho}}$ has as elements the twisted sectors of $I([\mathcal{M}_{1,i}/S_2])$ such that the distinguished automorphism of the twisted sectors exchanges two of the marked points, and such that the result is non-smoothable (so that they do actually come from the boundary, see Definition \ref{bordo}). See the last paragraph of \ref{appendicea} for a list of them.
   
   \end{example}
   
   \begin {remark} \label{nongenerdivis} Note that, as a consequence of this analysis, it is clear that in general if $\overline{X}$ is a twisted sector of $\overline{\mathcal{M}}_{2,n}$, its Chow group is not necessarily isomorphic to its cohomology group, and moreover neither its Chow ring nor its cohomology ring are generated by the divisor classes. In other words, the analogue of Theorem \ref{generatodivisori} is not true for the twisted sectors that come from the boundary. An example of this is the first graph in Example \ref{case1}. Some of these twisted sectors have $\overline{\mathcal{M}}_{1,n}$ as moduli space, and it is well known that the latter contains odd cohomology if $n \geq 11$.
   \end{remark}


\section{The cohomology of the inertia stacks of moduli of curves of genus $2$}
\label{cohomology}

In the previous section we have studied enough geometry of the inertia stacks of $\mathcal{M}_{2,n}$, $\mathcal{M}_{2,n}^{rt}$, and $\overline{\mathcal{M}}_{2,n}$ to compute the dimensions of the Chen--Ruan cohomology and stringy Chow vector spaces (Definition \ref{defcoomorb1}). The remainder of the section is devoted to writing down this information in a convenient, compact way.


\subsection{The dimension of the cohomology $H^*_{CR}(\mathcal{M}_{2,n})$ and $H^*_{CR}($\mdnrt$)$}

The rational Chow groups of the twisted sectors of $\mathcal{M}_{2,n}$ are all trivial. Indeed, we have proved in Section \ref{inertia2} that coarse space of each twisted sector of $\mathcal{M}_{2,n}$ is a quotient of a certain $\mathcal{M}_{0,k}$ by the action of a subgroup of the symmetric group $S_k$. 

Thus we have the following formula:
$$
\dim A^*_{st}(\mathcal{M}_{2,n}, \mathbb{Q})= \dim A^*(\mathcal{M}_{2,n}, \mathbb{Q}) + \textrm{ number of twisted sectors of } I(\mathcal{M}_{2,n}).
$$
The number on the right is equal to zero whenever $n\geq 7$, as we have already seen, and is equal to $(17,24,26,21,7,1,1)$ in the remaining seven cases.

The correction factor
$$
\tilde{h}_2(n):=\dim H^*_{CR}(\mathcal{M}_{2,n}, \mathbb{Q})- \dim H^*(\mathcal{M}_{2,n}, \mathbb{Q})
$$ 
corresponds to computing the invariant cohomology $H^*(\mathcal{M}_{0,k})^S$ for some $S < S_k$. Here we have the first seven values of $\tilde{h}_2(n)$ (they are zero afterwards): $(22,30,39,43,51,60,60)$.

Following \cite[Section 3]{pagani1}, we define the generating series of the dimensions of the cohomology vector spaces:

\begin{eqnarray}  \label{serietotale2} P_0(s):=\sum_{n=0}^{\infty}\frac{Q_0(n)}{n!} \ s^n, \\ P_{2,rt}(s):=\sum_{n=0}^{\infty}\frac{Q_{2,rt}(n)}{n!} \ s^n,  \\ P_{2,rt}^{CR}(s):=\sum_{n=0}^{\infty}\frac{Q_{2,rt}^{CR}(n)}{n!} \ s^n,
\end{eqnarray}

\noindent where:

\begin{eqnarray*}
Q_0(n):=\dim H^*(\overline{\mathcal{M}}_{0,n+1})=h(n), \\ 
Q_{2,rt} (n):=\dim H^*({\mathcal{M}}_{2,n}^{rt}), \\ 
Q_{2,rt}^{CR}(n):=\dim H^*_{CR}({\mathcal{M}}_{2,n}^{rt}),
\end{eqnarray*}

\noindent with the convention that $Q_0(0)=0$ and $Q_0(1)=1$.

Our Proposition \ref{twratcor}, together with the computation of the cohomology of the twisted sectors that we outlined in Section \ref{inertia}, gives the result below.

\begin {theorem} \label{samuel2thm} The following equality between power series relates the dimensions of the cohomology groups of $\overline{\mathcal{M}}_{0,n}$ and ${\mathcal{M}}_{2,n}^{rt}$ with the dimensions of the Chen--Ruan cohomology groups of ${\mathcal{M}}_{2,n}^{rt}$.
\begin {equation} \label {samuel2}
 P_{2,rt}^{CR}(s)=P_{2,rt}(s)+22+30P_0(s)+\frac{39}{2!} P_0(s)^2+\frac{43}{3!}P_0(s)^3+\frac{51}{4!}P_0(s)^4+\frac{60}{5!}P_0(s)^5+\frac{60}{6!}P_0(s)^6.
\end {equation} 
\end {theorem}

\begin{remark} A similar formula, with coefficients $(17,24,26,21,7,1,1)$, holds for the case of the stringy Chow group.
\end{remark}


\subsection{The dimension of the cohomology $H^*_{CR}($\mbdn$)$}

Here we want to write a formula similar to the one obtained in Equation \ref{samuel2} for the case of stable genus $2$ curves. Let us define the generating series of the dimensions of the cohomology groups:
\begin{eqnarray*}
Q_1(n):=\dim H^*(\overline{\mathcal{M}}_{1,n}), \\
Q_{2} (n):=\dim H^*(\overline{\mathcal{M}}_{2,n}), \\ 
Q_{2}^{CR}(n):=\dim H^*_{CR}(\overline{\mathcal{M}}_{2,n}),
\end{eqnarray*}
and then:
\begin{eqnarray}  \label{serietotale2stable} P_0'(s):= \sum_{n=0}^{\infty}\frac{Q_0(n+1)}{n!} \ {s^n}, \quad
P_1'(s):=\sum_{n=0}^{\infty}\frac{Q_1(n+1)}{n!} \ {s^n}, \\ P_{2}(s):=\sum_{n=0}^{\infty}\frac{Q_{2}(n)}{n!} \ s^n,  \quad \overline{P}_{2}^{CR}(s):=\sum_{n=0}^{\infty}\frac{Q_{2}^{CR}(n)}{n!} \ s^n.
\end{eqnarray}
Note that, with our convention, the degree zero term of $P_0'$ is $1$. The degree zero term of $P_1'$ is $2$. The power series $P_0'$ and $P_1'$ are just the total derivatives of $P_0$ and $P_1$.

\begin{theorem} \label{samuel2stabthm} The following equality between power series relates the dimensions of the cohomology groups of $\overline{\mathcal{M}}_{0,n}$, $\overline{\mathcal{M}}_{1,n}$, $\overline{\mathcal{M}}_{2,n}$ with the dimensions of the Chen--Ruan cohomology groups of $\overline{\mathcal{M}}_{2,n}$
\end{theorem}
\begin {equation} \label {samuel2stab}\begin{split}
 \overline{P}_{2}^{CR}= &P_2+32+43P_0+47\frac{P_0^2}{2!} +38\frac{P_0^3}{3!}+30\frac{P_0^4}{4!}+30\frac{P_0^5}{5!}+30 \frac{P_0^6}{6!}+  \\& +P_0'\left(43+52P_0+ 72 \frac{P_0^2}{2!}+ 40 \frac{P_0^3}{3!}+ 28 \frac{P_0^4}{4!} + 8 \frac{P_0^5}{5!} + 4 \frac{P_0^6}{6!} \right)+ \\ & +P_1' \left( 8+6 P_0+4\frac{P_0^2}{2!}+2 \frac{P_0^3}{3!}\right).
 \end{split}
 \end {equation}
\begin{proof}
The result is a sum of two contributions. The first one has the same form as Equation \ref{samuel2}. It is the cohomology of the compactification (see Definition \ref{compactifiedtwisted}) of the twisted sectors of $\mathcal{M}_{2,n}^{rt}$:
\begin {equation} \label {parziale1}
 \overline{P}_{2,rt}^{CR}:=P_2+29+39P_0+47\frac{P_0^2}{2!} +42\frac{P_0^3}{3!}+38\frac{P_0^4}{4!}+34\frac{P_0^5}{5!}+34\frac{P_0^6}{6!}.
\end {equation}
The second term is the cohomology of the twisted sectors of $\overline{\mathcal{M}}_{2,n}$ that come from the boundary (see Definition \ref{bordo}). We divide this second term into four terms, each one corresponding to the cohomology of the twisted sectors of one among the examples \ref{case1}, \ref{case2}, \ref{case3} and \ref{case4}.
The cohomology corresponding to the twisted sectors of Example \ref{case1} is given by:
\begin {equation} \label {parziale21} \begin{split}
 {U_1}:=&P_0'\left(37+48P_0+ 68 \frac{P_0^2}{2!}+ 40 \frac{P_0^3}{3!}+ 28 \frac{P_0^4}{4!} + 8 \frac{P_0^5}{5!} + 4 \frac{P_0^6}{6!} \right) + P_1' \left( 8+6 P_0+4\frac{P_0^2}{2!}+2 \frac{P_0^3}{3!}\right) + \\ & 
-\left(7+8P_0+ 14 \frac{P_0^2}{2!}+ 10 \frac{P_0^3}{3!}+ 10 \frac{P_0^4}{4!} + 4 \frac{P_0^5}{5!} + 4 \frac{P_0^6}{6!} \right).
\end{split}
\end {equation}
The cohomology corresponding to the twisted sectors of Example \ref{case2} is:
\begin {equation} \label {parziale22}
 {U_2}:=   8+ 6 P_0 + 6 \frac{P_0^2}{2!} .
\end {equation}
The cohomology corresponding to the twisted sectors of Example \ref{case3} is:
\begin {equation} \label {parziale23}
 {U_3}:= -2-2 P_0 -2 \frac{P_0^2}{2!} + P_0' \left(6+ 4 P_0 + 4\frac{P_0^2}{2!}\right).
\end {equation}
And, finally, the cohomology corresponding to the twisted sectors of Example \ref{case4} is:
\begin {equation} \label {parziale24}
 {U_4}:= 4+ 8 P_0 + 10\frac{P_0^2}{2!}+ 6\frac{P_0^3}{3!}+ 2\frac{P_0^4}{4!}.
\end {equation}
\noindent Summing everything, one obtains the desired result
$$
\overline{P}_2^{CR}= \overline{P}_{2,rt}^{CR} + U_1 + U_2 +U_3 +U_4.
$$
\end{proof}

\begin{remark} Equation \ref{samuel2stab} holds true after substituting $P_2^{CR}$ with the generating series of the dimensions of the rational stringy Chow groups, then modifying $P_1^{CR}$ in the same way, and replacing $P_2$ with the generating series of the dimensions of the rational Chow groups of $\overline{\mathcal{M}}_{2,n}$.
\end{remark}


\section {The age grading}
\label{grading}

In this section we define the grading on the Chen--Ruan cohomology groups. The Chen-Ruan cohomology turns out to be a Poincar\'e duality ring if the ordinary grading on the cohomology of the twisted sectors of the inertia stack is shifted by a suitable rational number (one for each twisted sector). This number is called  \emph{degree shifting number}, or \emph{fermionic shift}, or \emph{age}. In this section we define the age, and study it for the twisted sectors of $\mathcal{M}_{g,n}$ and $\mathcal{M}_{g,n}^{rt}$, assuming that it is known for the twisted sectors of $\mathcal{M}_g$. Then we write some explicit results for the case of genus $2$, pointed curves.


\subsection {Definition of Chen--Ruan degree}

We define the degree shifting number for the twisted sectors of the inertia stack of a smooth stack $X$. We denote the representation ring of $\mu_N$ by $R{\mu_N}$, and $\zeta_N$ is a choice of a generator for the group of the $N$-th roots of $1$. 

\begin {definition} (See \cite[Section 7.1]{agv2}.) A group homomorphism $\rho:\mu_N \to \mathbb{C}^*$ is determined by an integer $0 \leq k \leq N-1$ as $\rho( \zeta_N)= \zeta_N^k$. We define a function \emph{age}:
$$
\textrm{age}(\rho)=k/N.
$$
This function extends to a unique group homomorphism:
$$
\textrm{age}: R \mu_N \to \mathbb{Q}.
$$
\end {definition}

\noindent We now define the age of a twisted sector $Y$.
\begin{definition} \label{definitionage} (See \cite[Section 3.2]{chenruan}, \cite[Definition 7.1.1]{agv2}.) Let $Y$ be a twisted sector and $p$ a point of $Y$. It induces a morphism $p \to {I}(X)$, which is, according to Definition \ref{definertia}, a representable morphism $g:B \mu_N \to X$. Then the pull-back via $g$ of the tangent sheaf, $g^*(T_X)$, is a representation of $\mu_N$. 
 We define: $$a(Y):= \textrm{age}(g^*(T_X)).$$

\end{definition}

\noindent We can then define the orbifold, or Chen--Ruan, degree.

\begin{definition} \label{defcoomorb2} (See \cite[Definition 3.2.2]{chenruan}.) We define the \emph{$d$-th degree} Chen--Ruan cohomology group as follows:
$$
H^d_{CR}(X, \mathbb{Q}):= \bigoplus_i  H^{d-2 a(X_i,g_i)}(X_i, \mathbb{Q}),
$$
where the sum is over all twisted sectors. Analogously, the same shift is introduced in the stringy Chow ring (Definition \ref{defcoomorb1}).
\end{definition}

\begin {proposition} (See \cite[Lemma 3.2.1]{chenruan}, \cite[Theorem 7.4.1]{agv2}.) \label{codimension}Let $X_{(g)}$, $X_{(g^{-1})}$ be two connected components of the inertia stack of $X$ that are exchanged by the involution (Remark \ref{mappaiota}) of the inertia stack. Then if $c=\codim (X_{(g)},X)$, the following holds:
$$
a(X_{(g)})+ a(X_{(g^{-1})})= c.
$$
\end {proposition}

\begin{remark} \label{solonormale} If $Y$ is a twisted sector of the inertia stack of $X$, and $f:Y \to X$ is the restriction to $Y$ of the natural map $I(X) \to X$, then we have the following exact sequence:
$$
0 \to T_Y \to f^*(T_X) \to N_Y(X) \to 0
$$ 
that defines the normal bundle $N_Y(X)$. It follows from the definition of twisted sector that $a(Y)$ as defined in Definition \ref{definitionage} is equal to the age of the representation of $\mu_N$ on $N_Y(X)$.
\end{remark}


\subsection {Age of twisted sectors of \mgn \ and $\mathcal{M}_{g,n}^{rt}$}

The age for the twisted sectors of $\mathcal{M}_2$ can be computed using the fact that there is an explicit description of the fibers of the tangent bundle to the moduli stack of hyperelliptic curves. This is written down explicitly in \cite{spencer} and \cite{spencer2}, see also \cite{paganihyper} for the two missing twisted sectors $V.1$ and $V.2$.

We now establish two simple lemmas that allow the computation of the age for all the twisted sectors of $\mathcal{M}_{g,n}^{rt}$, assuming knowledge of the age of the twisted sectors of $\mathcal{M}_g$. A formula for the age of the twisted sectors of $\mathcal{M}_g$ is given in \cite{paganitommasi}.

\begin {lemma} \label{etalisci} Let $Y$ be a twisted sector of $\mathcal{M}_g$. If $Y(a_1,\ldots,a_{N-1})$ is a twisted sector of $\mathcal{M}_{g,n}$, obtained by adding marked points to $Y$ (\emph{cf.} Definition \ref{2admissible}), then the following relation holds between the ages of the two sectors:
\begin{equation}\label{legameeta}
a(Y(a_1, \ldots, a_{N-1})) = a(Y) + \frac{1}{N} \sum_{i=1}^{N-1}  \ \lambda(i) a_i,
\end{equation}
where $\lambda(i)$ is the inverse of $i$ in the group $\mathbb{Z}_N^*$.
\end{lemma}
\begin{proof} 
Let $(C, x_1, \ldots, x_n)$ be a pointed curve in $Y(a_1,\ldots,a_{N-1})$. The difference of the two ages in the statement is the age of the representation of $\mu_N$ on the tangent spaces $T_{x_k} C$. The computation then follows by our very construction (Definition \ref{settoretwistato}, Proposition \ref{instack2}). With our convention \ref{quasicanonica}, the action of the distinguished automorphism on the tangent space to a point of total ramification of the kind $i$ has weight the inverse of $i$ in the group $\mathbb{Z}_N^*$.
\end{proof}

\begin{definition} \label{points}(See \cite{kock} for more details.) Let $\mathbb{L}_i$ be the line bundle $s_i^*(\omega_{\pi})$ on $\overline{\mathcal{M}}_{g,k}$, where $\omega_{\pi}$ is the relative dualizing sheaf of the universal curve $\pi :\overline{\mathcal{C}}_{g,k} \to \overline{\mathcal{M}}_{g,k}$ and $s_i$ is the $i$-th section of the map $\pi$. These $\mathbb{L}_i$ are called \emph{line bundles of points} or \emph{cotangent line bundles}. We also define:
$$
\psi_i:= c_1(\mathbb{L}_i).
$$
\end{definition}

\begin{proposition} (See \cite{mumford} or \cite[Proposition 1.6]{getzler1} for the formulation given here.) \label{referenzaimpossibile} Let $G$ be a stable graph of genus $g$ and valence $n$, and let:
$$
p: \prod_{v \in V(G)} \overline{\mathcal{M}}_{g(v),n(v)} \to \overline{\mathcal{M}}_{g,n}
$$
be the ramified covering of the stratum $\overline{\mathcal{M}}(G)$ of $\overline{\mathcal{M}}_{g,n}$. Each edge of the graph determines two half-edges $s(e)$ and $t(e)$, and hence two line bundles $\mathbb{L}_{s(e)}$ and $\mathbb{L}_{t(e)}$ on $\prod_{v \in V(G)} \overline{\mathcal{M}}_{g(v),n(v)}$. The normal bundle $N_p$ is given by the formula:
$$
N_p= \bigoplus \mathbb{L}_{s(e)}^{\vee} \otimes \mathbb{L}_{t(e)}^{\vee}.
$$
\end{proposition}

\begin {corollary} \label{agerational} Let $Y(a_1,\ldots, a_{N-1})$ be a twisted sector of $\mathcal{M}_{g,k}$, and suppose that $I_1,\ldots, I_k$ is a partition of $[n]$ in $k$ non-empty subsets. The data of $Y$, $I_1, \ldots, I_k$ single out a twisted sector $X\cong Y \times \overline{\mathcal{M}}_{0,I_1+1} \times \ldots \times \overline{\mathcal{M}}_{0,I_k+1}$ of $\mathcal{M}_{g,n}^{rt}$ according to Proposition \ref{twratcor}. Let us call $\delta(a_i)$ the number of sets among $I_{1+\sum_{l<i} a_l}, \ldots, I_{\sum_{l \leq i} a_l}$ that contain exactly one element. Then the following equality holds:
$$
a(X)= a(Y(a_1, \ldots, a_{N-1})) +\frac{1}{N}\sum_{i=1}^{N-1} \ \lambda(i) \left(a_i-  \delta(a_i)\right),
$$
where $\lambda(i)$ is the multiplicative inverse of $i$ in $\mathbb{Z}_N^*$.
\end {corollary}

With the definitions given in this section, we can compute the orbifold Poincar\'e polynomials for $\mathcal{M}_{2,n}$.
If we define $Q_{2,sm}(n,m):=\dim H^{2 m}({\mathcal{M}}_{2,n})$, we can write:
\begin{eqnarray}   P_{2,sm}(s,t):=\sum_{n,m \geq 0}\frac{Q_{2,sm}(n,m)}{n!}s^n t^m  
\end{eqnarray}
and then, in analogy $Q^{CR}_{2,sm}(n,m):=\dim H^{2m}_{CR}({\mathcal{M}}_{2,n})$ 
\begin{eqnarray}   P_{2,sm}^{CR}(s,t):=\sum_{n,m \geq 0}\frac{Q^{CR}_{2,sm}(n,m)}{n!}s^n t^m . 
\end{eqnarray}
When the degree of the variable $s$ is greater than or equal to $6$, the power series $P_{2,sm}^{CR}$ coincides with $P_{2,sm}$. So we compute the seven non-trivial coefficients where the degree of $s$ is at most six, as polynomials in $t$: $$P_{2,sm}^{CR, (0)}(t),P_{2,sm}^{CR, (1)}(t),P_{2,sm}^{CR, (2)}(t),P_{2,sm}^{CR, (3)}(t),P_{2,sm}^{CR, (4)}(t),P_{2,sm}^{CR, (5)}(t),P_{2,sm}^{CR, (6)}(t).$$
\begin{theorem} \label{poincare2smooth} We compute the power series $P_{2,sm}^{CR}$ assuming knowledge of $P_{2,sm}$: 
\begin{equation} \begin{split} 
P_{2,sm}^{CR, (0)}(t)=&P_{2,sm}^{(0)}+ 1 +5 t^{\frac{1}{2}}+3 t + 2 t^{\frac{6}{5}}+ 2 t^{\frac{7}{5}} + t^{\frac{3}{2}}+2 t^{\frac{8}{5}}+2 t^{\frac{9}{5}}+ 3 t^2+ t^{\frac{5}{2}}, \\
P_{2,sm}^{CR, (1)}(t)=&P_{2,sm}^{(1)} + t^{\frac{1}{2}}+ t +  t^{\frac{9}{8}}+ 2 t^{\frac{6}{5}}+t^{\frac{5}{4}}+t^{\frac{4}{3}}+  t^{\frac{11}{8}}+  t^{\frac{7}{5}}+2 t^{\frac{8}{5}} + t^{\frac{13}{8}}+t^{\frac{5}{3}} + t^{\frac{7}{4}}+ t^{\frac{9}{5}}\\&+   t^{\frac{15}{8}} +5 t^2+ +t^{\frac{7}{3}}+  t^{\frac{12}{5}} +t^{\frac{8}{3}}+  t^{\frac{14}{5}}+ 5 t^3, \\
P_{2,sm}^{CR, (2)}(t)=&P_{2,sm}^{(2)} + t +   t^{\frac{3}{2}}+ t^{\frac{8}{5}}+  t^{\frac{11}{6}} + 9 t^2+  2 t^{\frac{11}{5}}+t^{\frac{7}{3}}+  t^{\frac{12}{5}}+  t^{\frac{5}{2}} +  t^{\frac{13}{5}}   +t^{\frac{8}{3}}+  2 t^{\frac{14}{5}}\\ & + 11 t^3+  t^{\frac{19}{6}}+t^{\frac{10}{3}}+   t^{\frac{17}{5}}+  t^{\frac{7}{2}}+t^{\frac{11}{3}}+t^4,\\
P_{2,sm}^{CR, (3)}(t)=&P_{2,sm}^{(3)} + t^{\frac{1}{2}}+5 t^{\frac{3}{2}}+3 t^{\frac{11}{5}}+3t^{\frac{7}{3}} +3t^{\frac{5}{2}}+3t^{\frac{13}{5}}+3t^{\frac{8}{3}}+ 6t^{\frac{10}{3}}+3t^{\frac{17}{5}}+4 t^{\frac{7}{2}} +6t^{\frac{11}{3}}+ 3 t^{\frac{19}{5}},\\
P_{2,sm}^{CR, (4)}(t)=&P_{2,sm}^{(4)} +t^2+12t^3+26t^4+12t^5,\\
P_{2,sm}^{CR, (5)}(t)=&P_{2,sm}^{(5)} +t^{\frac{5}{2}}+9t^{\frac{7}{2}}+26t^{\frac{9}{2}}+24t^{\frac{11}{2}},\\
P_{2,sm}^{CR, (6)}(t)=&P_{2,sm}^{(6)} +t^3+9 t^4 + 26 t^5 + 24 t^6.\\
\end{split}
\end{equation}
\end{theorem} 


\subsection {Age of twisted sectors of \mbdn}

As for the twisted sectors of $\overline{\mathcal{M}}_{2,n}$, those that are compactifications of twisted sectors of $\mathcal{M}_{2,n}^{rt}$ have degree shifting number that is simply equal to the open part that they compactify. Those coming from the boundary have been classified in four cases: see Examples \ref{case1}, \ref{case2}, \ref{case3} and \ref{case4}. From this, one can determine the orbifold Poincar\'e polynomials of $\overline{\mathcal{M}}_{2,n}$ defined as:
$$
\overline{P}_{2,n}^{CR}(t):=\sum_{m} \dim H^{2m}_{CR}(\overline{\mathcal{M}}_{2,n}) \  t^m.
$$
We write here the result for $n=0$. 

\begin{theorem} \label{poincare2} The orbifold Poincar\'e polynomial $\overline{P}_2^{CR}$ of $\overline{\mathcal{M}}_2$ equals:
$$
2+4 t^{\frac{1}{2}}+2 t^{\frac{3}{4}}+ 16 t+t^{\frac{7}{6}}+ 2 t^{\frac{6}{5}}+ 7 t^{\frac{5}{4}}+  t^{\frac{4}{3}}+ 2 t^{\frac{7}{5}}+ 23 t^{\frac{3}{2}}+ 2 t^{\frac{8}{5}}+ t^{\frac{5}{3}}+ 7 t^{\frac{7}{4}}+ 2 t^{\frac{9}{5}}+t^{\frac{11}{6}}+16t^2+2t^{\frac{9}{4}}+4 t^{\frac{5}{2}}+2 t^3.
$$
\end{theorem}
\noindent We could not find a compact way of writing the power series for general $n$.


\section {The stringy cup product}
\label{stringyproduct}

In this section we study the orbifold intersection theory on $\overline{\mathcal{M}}_{2,n}$. On the Chen--Ruan cohomology, defined in Definitions \ref{defcoomorb1}, \ref{defcoomorb2} as a graded vector space, there is a product that gives it a ring structure, which was first described in \cite[4.1]{chenruan}. This is also called stringy cup product. We review here the theory in the case of cohomology, but one can work in complete analogy with the Chow ring, as for example explained in \cite{agv1} and \cite{agv2}. In the first section we review the definition of Chen--Ruan cohomology as a graded algebra. Our main result of the last two sections is the computation of the top Chern class of the orbifold excess intersection bundle (a special instance of the virtual fundamental class in the case of Chen--Ruan cohomology), for the moduli stack of stable genus $2$ curves. This is a bundle on a disconnected stack, and in Theorem \ref{teochernclass}, we prove that either its top Chern class is $0$ or $1$, or it is possible to describe it in terms of first Chern classes of the line bundles of points $\mathbb{L}_i$ (see Definition \ref{points}).

\subsection{Definition}
The definition of the Chen--Ruan product involves the second inertia
stack.

\begin {definition} Let $X$ be an algebraic stack. The \emph{second
inertia stack} $I_2(X)$ is defined as:
$$ I_2(X)=I(X) \times_X I(X).$$
We will speak of the \emph{double untwisted sector} and of the \emph{double twisted sectors} of the second inertia stack (cf. Definition \ref{definertia}).
\end {definition}

\begin{remark} \label{liscezza2} An object in $I_2(X)$ is a triplet
$(x,g,h)$ where $x$ is an object of $X$ and $g, h \in$ Aut$(x)$. It can equivalently be given as $(x,g,h, (g h)^{-1})$. We observe that there is an isomorphism $\overline{\mathcal{M}}_{0,3}(X,0) \cong I_2(X)$ (\emph{cf.} the proof of \cite[Lemma 6.2.1]{agv2}). Therefore, being the first space smooth, the second inertia stack is also smooth (\emph{cf.} \ref{liscezza1}).
\end{remark}

\begin {remark} \label{doppinerzia} The stack $I_2(X)$ comes with three natural morphisms to
$I(X)$: $p_1$ and $p_2$, the two projections of the fiber product,
and $p_3$ which acts on objects sending $(x,g,h)$ to $(x,g h)$. This gives the following diagram, where $(Y,g,h,(gh)^{-1})$ is a double twisted sector and $(X_1,g)$, $(X_2,h)$, $(X_3, (gh))$ are twisted sectors:
\begin {equation}  \xymatrix{ & (X_1,g)\\
(Y,g,h) \ar@/^/[ur]^{p_1} \ar[r]^{p_2} \ar@/_/[dr]^{p_3}& (X_2,h)\\
& (X_3,g h).\\}\label {secondproj}
\end {equation}

\noindent 
\end{remark}
We review the definition of the excess intersection bundle over
$I_2(X)$, for $X$ an algebraic smooth stack. Let $(Y,g,h, (g h)^{-1})$ be a twisted sector of
$I_2(X)$. Let $H:= \langle g, h \rangle$ be the group
generated by $g$ and $h$ inside the automorphism group of a general object of $Y$. 

\begin{construction} \label{costruzione} \emph{
Let $\gamma_0, \gamma_1, \gamma_{\infty}$ be three small loops around $0, 1, \infty \subset \mathbb{P}^1$. Any map $\pi_1(\mathbb{P}^1 \setminus \{0,1, \infty\}) \to H$ corresponds to an $H$-principal bundle on $\mathbb{P}^1 \setminus \{0,1, \infty\}$.
Let $\pi^0 : C^0 \rightarrow \mathbb{P}^1 \setminus \{0, 1 , \infty\}$
be the $H$-principal bundle which corresponds to the map $\gamma_0 \to g, \gamma_1 \to h, \gamma_{\infty} \to (gh)^{-1}$. It can uniquely be extended to a ramified $H$-Galois covering $C \to \mathbb{P}^1$ (see \cite[Appendix]{fantechigottsche}), where $C$ is a smooth compact curve. Note that by definition $H$ acts on $C$ , and hence on $H^1(C, \mathcal{O}_C)$.
}\end{construction}

\noindent Let $f:Y \to X$ be the restriction of the canonical map $I_2(X) \to X$ to the twisted sector $Y$; then $H$ acts on $f^*(T_X)$, and acts trivially on $Y$.

\begin{definition} (See \cite{chenruan}.) \label{eccesso} With the same notation as in the previous
paragraph, the \emph{excess intersection bundle} over $Y$ is defined as:
$$
E_Y := \left(H^1(C, \mathcal{O}_C) \otimes_{\mathbb{C}} f^*(T_X) \right)^H,
$$
\emph{i.e.} the $H$-invariant subbundle of the expression between brackets.\footnote{There is also a purely algebraic definition of this excess intersection bundle, which avoids Construction \ref{costruzione}, see \cite[Section 4]{jarvis}.}
\end{definition}

\begin{remark}  The excess intersection bundle has different ranks on different connected components of $I_2(X)$. Moreover, since $H^1(C, \mathcal{O}_C)^H=0$, it is the same
to define $E_Y$ as:
$$
E_Y= \left(H^1(C, \mathcal{O}_C) \otimes N_YX \right)^H,
$$
where $N_YX$ is the coker of the canonical inclusion: $T_Y \to f^*(T_X)$ (\emph{cf.} Remark \ref{solonormale}).
\end{remark}
 We now review the definition of the Chen--Ruan product.

\begin{definition} \label{orbprod} Let $\alpha \in H^*_{CR}(X)$, $\beta \in
H^*_{CR}(X)$. We define:
$$
\alpha *_{CR} \beta =p_{3 *} \left( p_1^*(\alpha) \cup
p_2^*(\beta) \cup c_{top}(E) \right).
$$

\end{definition}
\noindent As announced, with this product, the Chen--Ruan cohomology becomes a graded algebra:
\begin {theorem}\label{prodotto} (See \cite{chenruan}.) With the grading defined in Definition \ref{defcoomorb2}, $(H^*_{CR}(X, \mathbb{Q}), *_{CR})$ is a graded $(H^*(X,\mathbb{Q}), \cup)$-algebra. 
\end {theorem}
This theorem implies that one can compute the rank of the excess intersection bundle in terms of the already computed age grading. If $(Y,(g,h,(gh)^{-1}))$ is a sector of the second inertia stack, the rank of the excess intersection bundle (here we stick to the notation introduced in Remark \ref{doppinerzia}) is:
\begin {equation}\label{formulaeccesso1}
\textrm{rk}(E_{(Y,g,h)})= a(X_1,g)+a(X_2,h)+a(X_3,(gh)^{-1})- \codim_Y X,
\end {equation}
where $\codim_Y X$ is $\dim X - \dim Y$.

\begin {corollary}\label{semplifica} The excess intersection bundle over double twisted sectors when either $g$,$h$, or $(gh)^{-1}$ is the identity, is the zero bundle.
\end {corollary}
Another formula that follows from Proposition \ref{codimension} relates the rank of the excess bundle over a double twisted sector and the rank of the excess bundle over the double twisted sector obtained by inverting the automorphisms that label the sector:
\begin {equation} \label{formulaeccesso2}
\textrm{rk}(E_{(Y,g^{-1},h^{-1})})= \codim_{X_1} X+ \codim_{X_2} X + \codim_{X_3} X -2 \codim_{Y} X - \textrm{rk}(E_{(Y,g,h)}).
\end {equation}


\subsection {The second inertia stack}
\label{secondinertia} 

The study of the second inertia stack $I_2(\overline{\mathcal{M}}_{2,n})$ in principle is similar to the study of the (first) inertia stack, which we carried out in Section \ref{inertia}. There is one important difference: the general element of a connected component of $I_2(\overline{\mathcal{M}}_{g})$ is a Galois covering with Galois group generated by two elements of finite order. Therefore in particular it need not be abelian, and the classification of \cite{pardini} cannot be used to give a modular description of these twisted sectors.
However the only point where we need $I_2(\overline{\mathcal{M}}_{2,n})$ is in the definition of the stringy cup product (Definition \ref{orbprod}), we will thus determine exactly what is essential for that formula.

Let us denote by $T^2_{g,k}$ the set of twisted sectors of the second inertia stack $I_2(\overline{\mathcal{M}}_{g,k}^{NR})$ (see Definition \ref{wrt} for the definition of $\overline{\mathcal{M}}_{2,n}^{NR}$). If $X$ is a double twisted sector of $I_2(\overline{\mathcal{M}}_{g,k}^{NR})$, whose general element corresponds to a curve $(C,x_1, \ldots, x_k)$ and two automorphisms $\phi_1$ and $\phi_2$ of it, let $S(X)$ be the subset of $[k]$ of the marked points where either $\phi_1$ or $\phi_2$ acts non-trivially (\emph{cf.} Remark \ref{lisciabile2} and Proposition \ref{pirttheorem}). Along the same lines that yielded to Propositions \ref{twratcor} and \ref{pirttheorem} in Section \ref{rationaltails}, it is not difficult to prove the following result.

\begin{proposition} \label{pirttheorem2} If $n>1$, the second inertia stack of $\overline{\mathcal{M}}_{g,n}$ is isomorphic to:
$$
I_2(\overline{\mathcal{M}}_{g,n})=\left( \overline{\mathcal{M}}_{g,n},1 \right) \coprod_{k=1}^{n} \coprod_{X \in T^2_{g,k}} X \times \coprod_{\{ I_s\} \in \mathcal{A}_{S(X),n}} \prod_{s \in S(X)} \overline{\mathcal{M}}_{0,I_s+1}.
$$ 
 \end{proposition}
 
\noindent Now $I_2(\overline{\mathcal{M}}_{g,k}^{NR})$ contains some connected components whose general element is a smooth genus $2$ curve, and others whose general element is singular. The first ones are compactifications of connected components of $I_2(\mathcal{M}_{2,n})$, for $n \leq 6$.
 
\begin{notation} \label{doppitwistati} Let $X_1$ and $X_2$ be two twisted sectors of $\overline{\mathcal{M}}_{g,k}$. We shall denote by $(X_1, X_2, X_3)$ the open and closed substack of the fiber product $X_1 \times_{\overline{\mathcal{M}}_{g,n}} X_2$ that maps onto the twisted sector $X_3$ under the third projection map $\iota \circ p_3$, where $\iota$ is the involution of Remark \ref{mappaiota}, and $p_3$ is the third projection of Definition \ref{doppinerzia}. If $X_i$ is the twisted sector whose distinguished automorphism is $g_i$, this convention is made \emph{ad hoc} to have the relation $g_1 g_2 g_3= 1$. This makes the computations of \ref{formulaeccesso1} and \ref{formulaeccesso2} more convenient.

If $I_1, \ldots, I_k$ is a partition of $[n]$ made of non-empty subsets, following Notation \ref{notazionemgnrt}, we denote by $(X_1, X_2, X_3)^{I_1,\ldots, I_k}$ a double twisted sector isomorphic to:
$$
(X_1, X_2, X_3) \times \overline{\mathcal{M}}_{0,I_1+1} \times \ldots \times \overline{\mathcal{M}}_{0,I_k+1},
$$
under the isomorphism of Proposition \ref{pirttheorem2}.
\end{notation} 

For later use, we shall need a few results on the double twisted sectors of the second inertia stack $I_2(\overline{\mathcal{M}}_2)$. We use the notation introduced in \ref{tabellona} and \ref{notazionecompsmooth} for the twisted sectors of $\overline{\mathcal{M}}_2$ whose general element is a smooth curve. It is easy to study the fiber product of each twisted sector with $\overline{\tau}$ (we refer the interested reader to \cite{spencer}), over $\overline{\mathcal{M}}_2$. For future reference, we report a few cases that will be of interests:

\begin{equation} \label{classetau} \overline{\tau} \times_{\overline{\mathcal{M}}_2} \overline{III}= \overline{VI}, \quad \overline{\tau} \times_{\overline{\mathcal{M}}_2} \overline{VI}= \overline{III}, \quad \overline{\tau} \times_{\overline{\mathcal{M}}_2} \overline{IV}= \overline{IV}, \quad \overline{\tau} \times_{\overline{\mathcal{M}}_2} \overline{II}= \overline{II}.
\end{equation}

We then study the fiber product of $\overline{III}$ with itself:

\begin{proposition} \label{terzoterzo} The fiber product $\overline{III} \times_{\overline{\mathcal{M}}_2} \overline{III}$ contains two one-dimensional connected components: $(\overline{III}, \overline{III}, \overline{III})$ and $(\overline{III}, \overline{III}, e)$. The projection map from both the two components onto the factor $\overline{III}$ induces an isomorphism.
\end{proposition}

\begin{proof} An object of $\overline{III}$ is a couple $(C, \alpha)$, where $C$ is a (family of) stable genus $2$ curves and $\alpha$ is an automorphism of order $3$ of $C$, such that the quotient $C/\langle \alpha \rangle$ is a genus $0$ curve with four points of ramification (see \ref{tabellona} and Definition \ref{compactifiedtwisted}). So there are two one-dimensional connected components of the double twisted sector $I_2(\overline{\mathcal{M}}_2)$ that are isomorphic to $(\overline{III}, \alpha, \alpha)$ and $(\overline{III}, \alpha, \alpha^2)$.
\end{proof}

\noindent Next, we can study the fiber product of $\overline{IV}$ with itself in complete analogy:

\begin{proposition} \label{quartoquarto} The fiber product $\overline{IV} \times_{\overline{\mathcal{M}}_2} \overline{IV}$ contains two one-dimensional connected components: $(\overline{IV}, \overline{IV}, \overline{\tau})$ and $(\overline{IV}, \overline{IV}, e)$. The projection map from  both the two components onto the factor $\overline{IV}$ induces an isomorphism.
\end{proposition}

\noindent The proof of this proposition is analogous to that of Proposition \ref{terzoterzo}. We shall also need a few consequences of these results, concerning the second inertia stacks
$I_2(\overline{\mathcal{M}}_{2,1})$ and $I_2(\overline{\mathcal{M}}_{2,2})$.

\begin{corollary}\label{terzoterzo1}The fiber product $\overline{III}_1 \times_{\overline{\mathcal{M}}_{2,1}} \overline{III}_1$ contains one one-dimensional connected component: $(\overline{III}_1, \overline{III}_1, \overline{III}_1)$. The projection map from it onto the factor $\overline{III}_1$ induces an isomorphism. The same result holds substituting $\overline{III}_1$ with $\overline{III}_{11}$, and $\overline{\mathcal{M}}_{2,1}$ with $\overline{\mathcal{M}}_{2,2}$.
\end{corollary}

\begin{corollary} \label{quartoquarto1} The fiber product $\overline{IV}_3 \times_{\overline{\mathcal{M}}_{2,1}} \overline{IV}_3$ contains one one-dimensional connected component: $(\overline{IV}_3, \overline{IV}_3, \overline{\tau}_1)$. The projection map from it onto the factor $\overline{IV}_3$ induces an isomorphism. The same result holds substituting $\overline{IV}_3$ with $\overline{IV}_{13}$ or $\overline{IV}_{31}$, and $\overline{\mathcal{M}}_{2,1}$ with $\overline{\mathcal{M}}_{2,2}$.
\end{corollary}
For later use, we remark that the one-dimensional stacks mentioned in Propositions \ref{terzoterzo}, \ref{quartoquarto} and Corollaries \ref{terzoterzo1} and \ref{quartoquarto1} have coarse moduli space isomorphic to $\mathbb{P}^1$.


\subsection{The excess intersection bundle}
\label{excessintersection}

We want to describe the excess intersection bundle $E_{2,n}$ on $I_2(\overline{\mathcal{M}}_{2,n})$. Assume that $(Y,H)$ is a connected component of $I_2(X)$, where $H$ is a group generated by two elements. We observe that there are two special cases:

\begin{enumerate}
\item The bundle $E$ on $Y$ has rank $0$. In this case, the top Chern class of $E$ on $Y$ is (by definition) equals $1$. If this is the case, we can say that on this component there is no orbifold excess intersection.
\item The bundle $E$ can have $0$ top Chern class. This occurs for instance when $\rk(E)> \dim(Y)$, or whenever $E$ contains a trivial subbundle. In this case we say that the orbifold excess intersection is trivial.
\end{enumerate}

For many double twisted sectors $Y$, formulas \ref{formulaeccesso1}, \ref{semplifica} and \ref{formulaeccesso2} can be used to show that the top Chern class of $E_{2,n}$ on $Y$ must be $0$ or $1$. We do not present here all these elementary computations. In this section we study the top Chern class of the excess intersection bundle on $I_2(\overline{\mathcal{M}}_{2,n})$, focusing on all the cases in which it is not $0$ or $1$. In this case we describe the excess intersection bundle in terms of the line bundles $\mathbb{L}_i$ (Definition \ref{points}). We make strong use of the notation that we introduced for the twisted sectors and double twisted sectors, see Notations \ref{notazionemgnrt} and \ref{doppitwistati}. 

Let us give a preview of the main results that are obtained in this section. \label{teochernclass} If $E_{2,n}$ has top Chern class that is not $0$ or $1$ on a certain connected component of $I_2(\overline{\mathcal{M}}_{2,n})$, then on that component it splits as a direct sum of line bundles. Here we list the only non-trivial cases:
\begin{enumerate}
\item If $n=0$, on one connected component in each fiber product $\overline{III}\times_{\overline{\mathcal{M}}_2} \overline{III}$, $\overline{III}\times_{\overline{\mathcal{M}}_2} \overline{VI}$ and $\overline{VI}\times_{\overline{\mathcal{M}}_2} \overline{VI}$. The excess bundle has rank $1$ and its first Chern class is described in Lemma \ref{treuno};
\item If $n>0$, on one connected component in the fiber products of each one of $\overline{III}_{1}^{[n]}$ and $\overline{III}_{11}^{I_1,I_2}$ with themselves. The excess bundle has rank $1$ and its first Chern class is described in Corollary \ref{tredue};
\item If $n>0$, on one connected component in the fiber product of each of $\overline{IV}_{3}^{[n]}$, $\overline{IV}_{13}^{I_1,I_2}$ and $\overline{IV}_{31}^{I_1,I_2}$ with themselves. The excess bundle has rank $2$ (resp. $3$) and can be written as a sum of line bundles. Its top Chern class is described in Proposition \ref{trequattro}; 
\item If $n>0$, on certain connected components in the fiber product of twisted sectors whose general element does not contain an irreducible curve of genus $2$. The description of the excess bundle $E_{2,n}$ on these components reduces to the description of $E_{1,n}$ (the excess intersection bundle for the Chen--Ruan cohomology of $\overline{\mathcal{M}}_{1,n}$), which was worked out in \cite[Section 6]{pagani1} (see Proposition \ref{tretre}). 
\end{enumerate} 
In these cases, we describe the top Chern classes in terms of $\psi$-classes on moduli spaces of stable genus $0$ or genus $1$ curves. The top Chern classes that are not of the kind $(1), (2), (3), (4)$ can be proved to be zero or one by a combined use of \ref{formulaeccesso1}, \ref{semplifica}, \ref{formulaeccesso2}, and Corollary \ref{corollariozero} below. We use Lemma \ref{lemmadecomp} to prove that the excess bundle always splits as a sum of line bundles.

We start by studying how the normal bundle $N_{I_2(X)} X$ behaves under forgetting rational tails, for general genus $g \geq 2$. Let $(Y,H)$ be a connected component of $I_2(\overline{\mathcal{M}}_{g})$, and suppose that $(Y_{\alpha(1), \ldots, \alpha(k)},H)$ is the connected component of $I_2(\overline{\mathcal{M}}_{g,k})$ that maps naturally into $\pi^*(Y)$:

$$\xymatrix{Y_{\alpha(1), \ldots, \alpha(k)} \ar[r]^i&\pi^*(Y) \ar@{}|{\square}[dr] \ar[r]\ar[d]&\overline{\mathcal{M}}_{g,k}\ar[d]^{\pi}\\
&Y \ar[r]&\overline{\mathcal{M}}_g.
}$$
If $I_1, \ldots,I_k$ is a partition of $[n]$, we can consider the component of the second inertia stack of $\overline{\mathcal{M}}_{g,n}$ obtained by gluing genus $0$ components to the marked points $\alpha(1), \ldots, \alpha(k)$, as in Proposition \ref{pirttheorem2}. We call it $Z:=(Y_{\alpha(1), \ldots, \alpha(k)}^{I_1, \ldots, I_k},H)$, it is defined by the upper Cartesian diagram:
\begin{equation} \label{diagrammone}
\xymatrix{Y_{\alpha(1),\ldots,\alpha(k)}^{I_1,\ldots,I_k} \ar@{}|{\square}[drr] \ar[rr]^{\hspace{-1cm}\phi} \ar[d]^p &&\overline{\mathcal{M}}_{g,k} \times \overline{\mathcal{M}}_{0,I_1 \sqcup \bullet_1} \times \ldots \times \overline{\mathcal{M}}_{0,I_k \sqcup \bullet_k} \ar[d] \ar[r]^{\hspace{1.8cm}j_{g,k}} & \overline{\mathcal{M}}_{g,n}  \\ Y_{\alpha(1), \ldots, \alpha(k)} \ar[r]^i&\pi^*(Y) \ar[r]^f\ar[d] \ar@{}|{\square}[dr]&\overline{\mathcal{M}}_{g,k}\ar[d]^{\pi}&\\
&Y \ar[r]^{\overline{f}}&\overline{\mathcal{M}}_g.&
}
\end{equation} 
We call $\mathbb{L}_{\bullet_i}$ the line bundle corresponding to the point $\bullet_i$ on $\overline{\mathcal{M}}_{0,I_i \sqcup \bullet_i}$ (see Definition \ref{points}).

\begin{lemma} \label{lemmadecomp} With the notation introduced above, the normal bundle $N_Z \overline{\mathcal{M}}_{g,n}$ splits as a representation of $H$:
\begin{displaymath}\begin{split} N_Z \overline{\mathcal{M}}_{g,n} =& (\pi \circ i \circ p)^*((N_Y \overline{\mathcal{M}}_g),\chi) \oplus (f \circ i \circ p)^* \left( (\mathbb{L}_1^{\vee}, \chi_1) \oplus \ldots \oplus (\mathbb{L}_k^{\vee}, \chi_k) \right) \\ &\oplus (\phi)^* \left( (\mathbb{L}_{\bullet_1}^{\vee}, \chi_1) \oplus \ldots \oplus (\mathbb{L}_{\bullet_k}^{\vee}, \chi_k) \right), \\\end{split}\end{displaymath}
for certain $\chi_1, \ldots, \chi_k$, one-dimensional characters of $H$ (see Definition \ref{points} for the line bundles $\mathbb{L}_i$).
\end{lemma}

\begin{proof} The proof of this lemma follows from the fact that Diagram \ref{diagrammone} is Cartesian, the vertical arrows are flat morphisms, and Proposition \ref{referenzaimpossibile}.
\end{proof}

\noindent This lemma reduces the computation of the top Chern class of $E_{g,n}$ to the corresponding computation for fiber products of twisted sectors of $\overline{\mathcal{M}}_{g,n}^{NR}$.
We state now this important straightforward consequence:

\begin{corollary} \label{corollariozero} Let $g>1$, and $(Y,H)$ be a connected component of $I_2(\overline{\mathcal{M}}_g)$ of dimension $0$. Then, let $(Y_{\alpha(1),\ldots,\alpha(k)}^{I_1,\ldots,I_k},H)$ be a corresponding component of $I_2(\overline{\mathcal{M}}_{g,n})$. The top Chern class of the excess intersection bundle $E_{g,n}$ restricted to the latter connected component is either $0$ or $1$.
\end{corollary}
\begin{proof}  Indeed, if any of the $\mathbb{L}_{\bullet_k}^{\vee}$ is in the excess intersection bundle, then $i^* f^* \mathbb{L}_k^{\vee}$, because $H$ acts on $\mathbb{L}_{\bullet_k}^{\vee}$ and $i^* f^* \mathbb{L}_k^{\vee}$ with the same character. Then the statement follows by observing that $i^* f^* \mathbb{L}_i^{\vee}$ is the trivial line bundle.
\end{proof}

\begin{remark} 
 The difference of the genus $1$ from the genus $g>1$ case is that in the base case of genus $1$, $\overline{\mathcal{M}}_{1,1}$ has one marked point. While if $g>2$ the line bundles constituting the normal bundle of the twisted sectors of $\overline{\mathcal{M}}_{g,n}$ come in pairs having the same character induced by the action of $H$, and one of the two in the couple is always trivial when the corresponding twisted sector in $\overline{\mathcal{M}}_g$ has dimension $0$, the case of a rational tail added to a twisted sector of $\overline{\mathcal{M}}_{1,1}$ of dimension $0$ does not fit into this picture. In fact, as follows from the main theorem in \cite[Section 6]{pagani1}, all the non-trivial orbifold excess intersections on \mbun \ are supported on double twisted sectors obtained by adding a single rational tail to a zero-dimensional twisted sector of $\overline{\mathcal{M}}_{1,1}$.
\end{remark} 

We can now go back to the case $g=2$ and study the cases when the top Chern class is neither $0$ nor $1$. 
After Lemma \ref{lemmadecomp} and Corollary \ref{corollariozero}, all we have to do is to study the excess intersection bundle on the double twisted sectors mentioned in Propositions \ref{terzoterzo}, \ref{quartoquarto} and Corollaries \ref{terzoterzo1}, \ref{quartoquarto1} and those whose general element is not a smooth curve. If $p$ is the class of a point in $\mathbb{P}^1$, $H^2(\mathbb{P}^1)$ is the one-dimensional rational vector space generated by $p$.

\begin{lemma} \label{treuno} With the previous identification, the excess intersection $E_2$ restricted to $(\overline{III}, \overline{III}, \overline{III})$ and $ (\overline{III}, \overline{VI}, \overline{VI})$ is a line bundle whose first Chern class is $\frac{1}{9}p$.
\end{lemma}

\begin{proof}
Using \eqref{formulaeccesso1} we can compute the rank of the excess intersection bundle restricted to these three components, and deduce that it is a line bundle in all three cases.

Using \ref{classetau}, we reduce the computation of $ (\overline{III}, \overline{VI}, \overline{VI})$ to the computation of the first Chern class of a line bundle on the double twisted sector $(\overline{III}, \overline{III}, \overline{III})$.

Let $f: \overline{III} \to \overline{\mathcal{M}}_2$ be the natural forgetful map. On $\overline{III}$ there is an exact sequence (that defines the vector bundle $\coker$):
 $$
 0 \to T_{\overline{III}} \to f^* T_{\overline{\mathcal{M}}_2} \to \coker \to 0.
 $$
By construction, $T_{\overline{III}}$ is equal to $\left(f^*(T_{\overline{\mathcal{M}}_2})\right)^{\mu_3}$, the $\mu_3$-invariant part of the pull-back. Symmetry arguments lead to the fact that $\coker$ splits as a sum of isomorphic line bundles, each one carrying one of the two non-trivial representations of $\mu_3$. We want to show that the first Chern class of the rank $3$ bundle $f^* T_{\overline{\mathcal{M}}_2}$ is $\frac{1}{3}$, and that the first Chern class of the line bundle $T_{\overline{III}}$ is $\frac{1}{9}$. From this we can deduce that the degree of each line bundle in the decomposition of $f^* T_{\overline{\mathcal{M}}_2}$ on the family $\overline{III}$ is equal to $\frac{1}{9}$.

To compute the first Chern class of $f^* T_{\overline{\mathcal{M}}_2}$ on $\overline{III}$ is equivalent to computing the degree of $-K_{\overline{\mathcal{M}}_2}$, the anti-canonical class. Using the relations established in \cite[Part III]{mumford} (we also refer to that paper for the definition of the classes $\delta_1$ and $\lambda$):
$$
K_{\overline{\mathcal{M}}_2}= - 7 \lambda + 2 \delta_1.
$$
Now one can see that the degree of $\lambda$ on $\overline{III}$ is equal to $\frac{1}{18}$ and the degree of $\delta_1$ is $\frac{1}{36}$. This gives $c_1(f^* T_{\overline{\mathcal{M}}_2}) = \frac{1}{3}$.

We now compute the degree of the tangent line bundle $T_{\overline{III}}$. We have seen that $\overline{III}$ is a stacky $\mathbb{P}^1$ with generic stabilizer $\mu_6$. A more detailed analysis shows that there are two stacky points, with stabilizer of order $12$ and $36$. It is well known that the degree of the tangent line bundle to such stacky $\mathbb{P}^1$ is equal to $\frac{1}{12} + \frac{1}{36}= \frac{1}{9}$.
\end{proof}

Let us identify the rational cohomologies of $(\overline{III}_1, \overline{III}_1, \overline{III}_1)^{[n]}$ and $(\overline{III}_{11}, \overline{III}_{11}, \overline{III}_{11})^{I_1,I_2}$  with $H^*(\mathbb{P}^1) \otimes H^*(\overline{\mathcal{M}}_{0,I_1+1}) \otimes H^*(\overline{\mathcal{M}}_{0,I_2+1})$ under the K\"unneth decomposition (and similarly with the other double twisted sectors obtained by adding rational tails).

\begin{corollary} \label{tredue} The top Chern class of $E_{2,n}$ is $\frac{1}{9}p \otimes 1$ on $(\overline{III}_1, \overline{III}_1, \overline{III}_1)^{[n]}$ and $\frac{1}{9}p \otimes 1 \otimes 1$ on $(\overline{III}_{11}, \overline{III}_{11}, \overline{III}_{11})^{I_1,I_2}$.
\end{corollary}

Let us now describe the top Chern class on the double twisted sectors of Corollary \ref{quartoquarto1}.

\begin{proposition}\label{trequattro} The top Chern class of $E_{2,1}$ on $(\overline{IV}_3, \overline{IV}_3, \tau_1)$ is $-\frac{1}{8} p$.  The top Chern class of $E_{2,n}$ is $-\frac{1}{8}p \otimes - \psi_{\bullet}$ on $(\overline{IV}_3, \overline{IV}_3, \tau_1)^{[n]}$ and $-\frac{1}{8}p \otimes - \psi_{\bullet_1} \otimes -\psi_{\bullet_2}$ on $(\overline{IV}_{13}, \overline{IV}_{13}, \overline{\tau}_{11})^{I_1,I_2}$ and $(\overline{IV}_{31}, \overline{IV}_{31}, \overline{\tau}_{11})^{I_1,I_2}$.
\end{proposition}

\begin{proof} The first statement follows from the fact that the degree of $\psi$ classes on the moduli stack $(\overline{IV}, \overline{IV}, \tau)$ is equal to $\frac{1}{8}$. The second is a consequence of Lemma \ref{lemmadecomp}. 
\end{proof}

Finally, we study the top Chern classes of $E_{2,n}$ on $(Y,H)$, where the general element of $Y$ is a genus $2$ curve whose general element does not contain an irreducible genus $2$ component. These twisted sectors are described in Examples \ref{case1}, \ref{case2}, \ref{case3} and \ref{case4}. Consider $X_1$ and $X_2$ two such twisted sectors, and apply \ref{formulaeccesso1} and \ref{formulaeccesso2}. Then one sees that the top Chern class of $E_{2,n}$ is always $0$ or $1$ on the component $X_1 \times_{\overline{\mathcal{M}}_{2,n}} X_2$, unless $X_1$ and $X_2$ are twisted sectors of $I(\overline{\mathcal{M}}_{1,1}) \times I(\overline{\mathcal{M}}_{1,n+1})$ obtained by gluing the last two marked points. Moreover, in this case, the results from \cite[Section 6]{pagani1}, lead to the fact that the top Chern classes of $E_{2,n}$ can be different from $0$ or $1$ only on products of twisted sectors whose associated graph is among the following.
\begin{equation}
\begin{tikzpicture}[baseline]
      \path(0,0) ellipse (2 and 1);
      \tikzstyle{level 1}=[counterclockwise from=-60,level distance=9mm,sibling angle=120]
      \node (A0) at (0:1) {$\scriptstyle{T_1}$};
      \tikzstyle{level 1}=[counterclockwise from=120,level distance=9mm,sibling angle=20]
      \node (A1) at (180:1) {$\scriptstyle{1_n}$} child child child child child child child;
      \path (A0) edge [bend left=0] (A1);
    \end{tikzpicture}
     \begin{tikzpicture}[baseline]
      \path(0,0) ellipse (2 and 2);
      \node (A0) at (0:1) {$\scriptstyle{T_1}$};
      \node (A1) at (240:1) {$\scriptstyle{T_1}$};
      \tikzstyle{level 1}=[counterclockwise from=75,level distance=9mm,sibling angle=15]
\node (A2) at (120:1) {$\scriptstyle{0_{n}}$} child child child child child child;
      \path (A0) edge [bend left=0] (A2);
      \path (A1) edge [bend left=0] (A2);
    \end{tikzpicture}
 \begin{tikzpicture}[baseline]
      \path(0,0) ellipse (2 and 2);
      \tikzstyle{level 1}=[counterclockwise from=-30,level distance=9mm,sibling angle=15]
      \node (A0) at (0:1) {$\scriptstyle{T_2}$} child{[fill] circle (2pt)};
      \node (A1) at (240:1) {$\scriptstyle{T_1}$};
      \tikzstyle{level 1}=[counterclockwise from=75,level distance=9mm,sibling angle=15]
\node (A2) at (120:1) {$\scriptstyle{0_{n}}$} child child child child child child;
       \path (A0) edge [bend left=0] (A2);
      \path (A1) edge [bend left=0] (A2);
    \end{tikzpicture}
  \begin{tikzpicture}[baseline]
      \path(0,0) ellipse (2 and 2);     
      \tikzstyle{level 1}=[counterclockwise from=-45,level distance=9mm,sibling angle=30]
      \node (A0) at (0:1) {$\scriptstyle{T_3}$} child{[fill] circle (2pt)} child{[fill] circle (2pt)};
      \node (A1) at (240:1) {$\scriptstyle{T_1}$};
      \tikzstyle{level 1}=[counterclockwise from=75,level distance=9mm,sibling angle=15]
\node (A2) at (120:1) {$\scriptstyle{0_{n}}$} child child child child child child;
      \path (A0) edge [bend left=0] (A2);
      \path (A1) edge [bend left=0] (A2);
    \end{tikzpicture}    
 \begin{tikzpicture}[baseline]
      \path(0,0) ellipse (2 and 2);
      \tikzstyle{level 1}=[counterclockwise from=-60,level distance=9mm,sibling angle=30]
      \node (A0) at (0:1) {$\scriptstyle{T_4}$}child{[fill] circle (2pt)} child{[fill] circle (2pt)} child{[fill] circle (2pt)};
      \node (A1) at (240:1) {$\scriptstyle{T_1}$};
      \tikzstyle{level 1}=[counterclockwise from=75,level distance=9mm,sibling angle=15]
\node (A2) at (120:1) {$\scriptstyle{0_{n}}$} child child child child child child;
      \path (A0) edge [bend left=0] (A2);
      \path (A1) edge [bend left=0] (A2);
    \end{tikzpicture}
     \end {equation}
\begin{proposition} \label{tretre} (See \cite[Section 6]{pagani1}.) Let $(Y,H)$ be a connected component of $I_2(\overline{\mathcal{M}}_{2,n})$, whose general element does not contain an irreducible genus $2$ component. Then the top Chern class of $E_{2,n}$ on $(Y,H)$ can be different from $0$ or $1$ only if $(Y,H)$ is a connected component of $I_2(\overline{\mathcal{M}}_{1,1}) \times I_2(\overline{\mathcal{M}}_{1,n+1})$. In this case the top Chern class is different from $0$ or $1$ exactly when the first coordinate is among: $(C_4,\langle i,i \rangle),(C_6,\langle \epsilon^2,\epsilon^2 \rangle), (C_6,\langle \epsilon,\epsilon^2 \rangle)$ or the second coordinate is among: $(C_4^{[n]},\langle i,i \rangle),(C_6^{[n]},\langle \epsilon^2,\epsilon^2 \rangle), (C_6^{[n]},\langle \epsilon,\epsilon^2 \rangle)$.
\end{proposition}

In this last case, the top Chern class is a $-\psi$ class on the gluing point(s), as shown in \cite[Section 6]{pagani1}.


\section {The Chen--Ruan cohomology as an algebra on the ordinary cohomology}
\label{algebra}

In this last section we study the generators of the Chen--Ruan cohomology as an algebra on the ordinary cohomology ring. We accomplish this task for the even part of the orbifold cohomology.

\begin{definition} \label{evenodd} Let $X$ be a Deligne--Mumford stack. We define the \emph{even} and \emph{odd} parts of the Chen--Ruan cohomology of $X$ as: $$H^{ev}_{CR}(X):=H^{ev}(I(X)), \  H^{odd}_{CR}(X):=H^{odd}(I(X))$$
where the grading is the usual one, \emph{i.e.} the grading is \emph{not} shifted by the (age) degree shifting number.
\end{definition} 

The even Chen--Ruan cohomology $H^{ev}_{CR}(X)$ is then naturally an $H^{ev}(X)$-algebra (Definition \ref{prodotto}). The main purpose of this last section will be to study the generators of the algebra $H^{ev}_{CR}(\overline{\mathcal{M}}_{2,n})$ over the ring $H^{ev}(\overline{\mathcal{M}}_{2,n})$. The main result is Theorem \ref{positivo}, where we show that the even Chen--Ruan cohomology ring of the moduli stack of stable pointed genus $2$ curves is generated multiplicatively by the fundamental classes of the twisted sectors and some classes that we are going to define in Notation \ref{classispeciali}. This theorem depends upon Conjecture \ref{getzlerremark} by Getzler, which we now review.

We start with a brief survey on the tautological ring, as defined by Faber--Pandharipande in their paper \cite{faberpanda}:

\begin{definition} \label{tautolofaber} (See \cite[0.1]{faberpanda}.) The \emph{system of tautological rings} is defined to be the set of smallest $\mathbb{Q}$-subalgebras of the Chow rings, $$R^*(\overline{\mathcal{M}}_{g,n}) \subset A^*(\overline{\mathcal{M}}_{g,n})$$
satisfying the following two properties: \begin{enumerate} \item The system is closed under push-forward via all maps forgetting markings; \item The system is closed under push-forward via all gluing maps. \end{enumerate}
\end{definition}
We here report the part of Getzler's conjectures that will be used to prove some of the results in the following sections:
\begin{conjecture} \label{getzlerremark} (See \cite[p.2]{graberpanda}, \cite[p.1]{getzler1}.) The cycle map $R^*(\overline{\mathcal{M}}_{1,n}) \to H^{ev}(\overline{\mathcal{M}}_{1,n})$ is surjective.
\end{conjecture}
We use in this paper the content of this conjecture in the proof of Corollary \ref{corollariogetzler}. We will mark with the symbol $*$ the results that depend upon Conjecture \ref{getzlerremark}.

The injectivity of the cycle map is the other part of Getzler's conjecture. The injectivity of this map follows from a more general conjecture, which goes under the name of \emph{Gorenstein conjecture}, and is due to Faber and Hain--Looijenga.
\begin{conjecture} \label{gorensteinconj} (See \cite{faberhain}) The tautological ring $R^*(\overline{\mathcal{M}}_{g,n})$ is a Poincar\'e duality ring, with socle in top degree $3g-3+n$. 
\end{conjecture}


\subsection{Pull-Backs of classes to the twisted sectors}

In this section, if $X$ is a twisted sector of the inertia stack of $\overline{\mathcal{M}}_{g,n}$, we call $f:X \to \overline{\mathcal{M}}_{g,n}$ the natural forgetful map, and $f^*$ the morphism induced in cohomology. If $\alpha \in H^*(\overline{\mathcal{M}}_{g,n})$, it follows from Definition \ref{prodotto} that:
$$
\alpha *_{CR} 1_X = f^*(\alpha),
$$
if $1_X$ is the fundamental class of the twisted sector $X$ inside the Chen--Ruan cohomology of $\overline{\mathcal{M}}_{g,n}$. From this, the importance of studying the surjectivity of the map $f^*$. 

If $X$ is such a twisted sector and $f$ the natural forgetful map, in the spirit of Notation \ref{notazionemgnrt}, we call $X^{I_1,\ldots,I_k}$ the twisted sector obtained by adding rational tails, and $f^{I_1, \ldots, I_k}$ the corresponding forgetful map to $\overline{\mathcal{M}}_{g,n}$.

\begin{lemma} \label{cited} Suppose that $X$ is a twisted sector of $\overline{\mathcal{M}}_{g,k}^{NR}$, and $I_1, \ldots, I_k$ is a partition of $[n]$. Then the pull-back in cohomology $f_{I_1,\ldots,I_k}^*$ is surjective iff the pull-back $f^*$ is surjective. The pull-back $f_{I_1,\ldots,I_k}^*$ surjects onto the even cohomology iff the pull-back $f^*$ surjects onto the even cohomology. 
\end{lemma}
\begin{proof} (This fact was recognized in the genus $1$ case in \cite[Section 7]{pagani1}) One applies the K\"unneth decomposition to the cohomology of $X^{I_1,\ldots,I_k}$. The surjectivity over the cohomology of the $l$-th rational tail is obtained by observing that given any partition $P_l$ of $I_l+1$ in two subsets, there is a divisor on \mbgn \, whose inverse image on $X^{I_1,\ldots,I_k}$ corresponds to separating the points of the $l$-th rational tail according to $P_l$.
\end{proof}

\begin{proposition} \label{suriettivita1} The pull-back map $f^*$ is surjective onto the cohomology of all the twisted sectors $X$ of $\overline{\mathcal{M}}^{NR}_{2,k}$, with the exception of the twisted sectors of Example \ref{case1} whose dual graph contains a vertex of genus $1$, and of the twisted sectors $\overline{II}, \overline{II}_{1}, \overline{II}_{11}$.
\end{proposition}
\begin{proof} The result is trivial when the dimension $\dim(X)$ is $0$. When $\dim(X)=1$ the coarse space of ${X}$ is $\mathbb{P}^1$, and the result follows from the fact that all such $X$ intersect the boundary at a finite number of points.  

Let ${X}$ be one of the twisted sectors $\overline{\tau}, \overline{\tau}_1, \ldots, \overline{\tau}_{111111}$. According to Theorem \ref{generatodivisori}, it is enough to show that $f^*$ is surjective on the divisor classes. Each divisor $D_I$, for $I \subset [n]$, pulls-back to a multiple of a divisor class in ${X}$, and all of the divisor classes of ${X}$ are possibly multiples of such pull-backs.

So we are left with the classes that come from the boundary, discussed in Examples \ref{case1}, \ref{case2}, \ref{case3} and \ref{case4}. The twisted sectors of dimension $>1$ that are not in the first line of the first set of figures in Example \ref{case1} have cohomology generated by the divisor classes, and one can prove surjectivity in analogy with Lemma \ref{cited}.  
\end{proof}

The twisted sectors whose dual graph contains a vertex of genus $1$ are those pictured in the first line of the first set of figures in Example \ref{case1}. 
To prove the surjectivity claim onto the even part, we use a result of Belorousski. First we recall a definition.

\begin{definition} (See \cite[p.2]{getzler2}.) Let $G$ be a stable graph; we will say it is a \emph{necklace}, if it has a single circuit, all of whose vertices have genus $0$. A \emph{necklace cycle} is the class of the locus whose general element is a curve with a necklace $G$ as its dual graph.
\end{definition}

\begin{proposition} (See \cite{belo}.) \label{belor} Two sets of generators for $R^*(\overline{\mathcal{M}}_{1,n})$ are:
\begin{enumerate}
\item The boundary strata classes.
\item All products of divisor classes, and the necklace cycles.
\end{enumerate}
Moreover the cycle map $R^*($\mbun$) \to H^*($\mbun$)$ is an isomorphism when $n \leq 10$.
\end{proposition}

Let us now consider the two substacks $C_{n+1}$ and $D_{n+1}$ of $\overline{\mathcal{M}}_{2,n}$ whose dual graphs correspond respectively to the graphs:
$$
\begin{tikzpicture}[baseline]
      \path(0,0) ellipse (2 and 1);
      \tikzstyle{level 1}=[counterclockwise from=-60,level distance=9mm,sibling angle=120]
      \node (A0) at (0:1) {$\scriptstyle{{\hspace{0.08cm}}_1^{\hspace{0.2cm} }}$};
      \tikzstyle{level 1}=[counterclockwise from=120,level distance=9mm,sibling angle=20]
      \node (A1) at (180:1) {$\scriptstyle{1_n}$} child child child child child child child;

      \path (A0) edge [bend left=0] (A1);
    \end{tikzpicture}
    \begin{tikzpicture}[baseline]
      \path(0,0) ellipse (2 and 1);
      \tikzstyle{level 1}=[counterclockwise from=-60,level distance=9mm,sibling angle=120]
      \node (A0) at (0:1) {$\scriptstyle{{\hspace{0.08cm}}_0^{\hspace{0.2cm} }}$};
      \tikzstyle{level 1}=[counterclockwise from=120,level distance=9mm,sibling angle=20]
      \node (A1) at (180:1) {$\scriptstyle{1_n}$} child child child child child child child;

      \path (A0) edge [bend left=0] (A1);
    \draw (A0) .. controls +(-15:1.2) and +(15:1.2) .. (A0);
    \end{tikzpicture},
   $$
which we call $G_{n+1}$ and $H_{n+1}$. 

\begin{lemma} \label{getzlerlemma} The inclusion map $i: D_{n+1} \to \overline{\mathcal{M}}_{2,n}$ induces a surjective pull-back $i^*$ on the divisor and necklace cycle classes of $D_{n+1}$. The same result holds for $C_{n+1}$. 
\end{lemma}
\begin{proof}

We have to show that any divisor and any necklace can be obtained by intersecting $D_{n+1}$ with certain boundary strata cycles in $\overline{\mathcal{M}}_{2,n}$. For the combinatorics of the intersection of boundary strata classes we refer to \cite[Appendix]{graberpanda} and to \cite{stephanie}. This involves the terminology of $(G,H)$-structure on a given graph, which we are about to use.

Let us consider an arbitrary divisor in $D_{n+1}$. A divisor in $D_{n+1}$ corresponds to a partition $I_1 \sqcup I_2$ of $[n]$:
 $$ \begin{tikzpicture}[baseline]
      \path(0,0) ellipse (2 and 2);
      \node (A0) at (0:1) {$\scriptstyle{{\hspace{0.08cm}}_0^{\hspace{0.2cm} }}$};
      \tikzstyle{level 1}=[counterclockwise from=90,level distance=9mm,sibling angle=60]
      \node (A1) at (120:1) {$\scriptstyle{1_{I_1}}$} child child child;
      \tikzstyle{level 1}=[counterclockwise from=180,level distance=9mm,sibling angle=60]
      \node (A2) at (240:1) {$\scriptstyle{0_{I_2}}$} child child child;

      \draw (A0) .. controls +(-15:1.2) and +(15:1.2) .. (A0);
      \path (A0) edge [bend left=0] (A1);
      \path (A1) edge [bend left=0] (A2);
    \end{tikzpicture}
 $$  
This is the only graph that admits a generic $(G_{n+1},G_{I_1,I_2})$-structure, where $G_{I_1,I_2}$ is the graph:
$$\begin{tikzpicture}[baseline]
      \path(0,0) ellipse (2 and 1);
      \tikzstyle{level 1}=[counterclockwise from=-60,level distance=9mm,sibling angle=60]
      \node (A0) at (0:1) {$\scriptstyle{0_{I_2}}$} child child child ;
      \tikzstyle{level 1}=[counterclockwise from=120,level distance=9mm,sibling angle=60]
      \node (A1) at (180:1) {$\scriptstyle{2_{I_1}}$} child child child;

      \path (A0) edge [bend left=0] (A1);
    \end{tikzpicture}.
$$
Analogously, if $B_{I_1}$ is a necklace with marked points in $I_1$, the dual graph of a necklace cycle in $D_{n+1}$ looks like:
$$
\begin{tikzpicture}[baseline]
      \path(0,0) ellipse (2 and 1);
      \tikzstyle{level 1}=[counterclockwise from=-60,level distance=9mm,sibling angle=120]
      \node (A0) at (0:1) {$\scriptstyle{B_{I_1}}$};
      \tikzstyle{level 1}=[counterclockwise from=120,level distance=9mm,sibling angle=20]
      \node (A1) at (180:1) {$\scriptstyle{1_{I_2}}$} child child child child child child child;

      \path (A0) edge [bend left=0] (A1);
    \end{tikzpicture};
$$
and this is the only graph that admits a generic $(G_{n+1}, N_{I_1,I_2})$-structure, where $N_{I_1,I_2}$ is the graph obtained from the latter graph by contracting the only edge that is represented in the picture.
\end{proof}
After this lemma and Proposition \ref{belor}, we see that the statement of Proposition \ref{suriettivita1} extends to the twisted sectors of Example \ref{case1} whose dual graph contains a vertex of genus $1$ and less than $11$ marked points. By assuming Conjecture \ref{getzlerremark}, we can extend the statement to the case of even cohomology.
\begin{corollarystar} \label{corollariogetzler} (See Conjecture \ref{getzlerremark}.) The pull-back map in cohomology $f^*$ is surjective onto the \emph{even} cohomology of all the twisted sectors of $\overline{\mathcal{M}}^R_{2,k}$, apart from $\overline{II}, \overline{II}_{1}, \overline{II}_{11}$.
\end{corollarystar}

Now we study the cases of the pull-back via $f$ to the twisted sectors $\overline{II}, \overline{II}_1$ and $\overline{II}_{11}$. We study in detail the case of $\overline{II}$, as the others follow similarly. 
We follow the analysis of \cite[Lemma 3.7.0.2]{spencer}. 
Note that by Lemma \ref{duezero}, the rational Chow group and the rational cohomology agree for this stack. Let us consider the quotient map $\pi: \overline{\mathcal{M}}_{0,5} \to [\overline{\mathcal{M}}_{0,5}/S_3]$. There are four cycle classes in the latter stack, which we call $\mathcal{A}, \mathcal{B}, \mathcal{C}, \mathcal{D}$ (following Spencer's notation), defined by: $$\pi^* \mathcal{A}:= 2D_{1,2}+2D_{1,3}+2D_{2,3}, \ \pi^*\mathcal{B}:= D_{1,4}+D_{2,4}+D_{3,4}, \ \pi^*\mathcal{C}:= D_{1,5}+D_{2,5}+D_{3,5}, \ \pi^*\mathcal{D}:= D_{4,5},$$
where $D_{i,j}$ is the divisor in $\overline{\mathcal{M}}_{0,5}$ whose general element is a reducible genus $0$ curve with two smooth components, one with marked points $i,j$. As the relations in $\overline{\mathcal{M}}_{0,5}$ are all known from \cite{keel}, one obtains with some linear algebra the relation:
\begin{equation} \label{relazione}
6 \mathcal{D} + \mathcal{A}= 2 (\mathcal{B} + \mathcal{C}).
\end{equation}
Thus, we have:

\begin{proposition} \label{generazione} The rational Picard group of $\overline{II}$ is freely generated by the three classes, $\mathcal{A}$, $\mathcal{D}$ and $\mathcal{B}-\mathcal{C}$. 
\end{proposition}
We observe that the two classes $\mathcal{B}$ and $\mathcal{C}$ are exchanged by the action of $S_2$ on $[\overline{\mathcal{M}}_{0,5}/S_3]$, which exchanges the last two marked points. This fact will play a role in Proposition \ref{corollariocoker}. So let now $f: \overline{II} \to \overline{\mathcal{M}}_2$ be the restriction of the map from the inertia stack.

\begin{proposition} \label{corollariocoker} The class $\mathcal{B}- \mathcal{C}$ is not in the image of $f^*:A^1(\overline{\mathcal{M}}_2) \to A^1(\overline{II})$, moreover it generates the cokernel of the latter map. 
\end{proposition}
\begin{proof} 
We start by showing that the class $\mathcal{B}- \mathcal{C}$ is not in the image of $f^*$. Let $\overline{B}_2 \subset \overline{\mathcal{M}}_2$ be the moduli stack of bielliptic curves of genus $2$. Then the map $f: \overline{II} \to \overline{\mathcal{M}}_2$ factors via the inclusion $i: \overline{B}_2 \to \overline{\mathcal{M}}_2$: $f=i \circ g$. The resulting map $g: \overline{II} \to \overline{B}_2$ forgets the bielliptic involution. A proof analogous to that of Proposition \ref{duezero} shows that $\overline{B}_2$ has the same coarse moduli space as $[\overline{\mathcal{M}}_{0,5}/S_3 \times S_2]$. So we have a commutative diagram:
$$
\xymatrix{\overline{II} \ar[d] \ar[r]^{g}& \overline{B}_2 \ar[d] \ar[r]^i& \overline{\mathcal{M}}_2 \\
[\overline{\mathcal{M}}_{0,5}/S_3] \ar[r]^{\hspace{-0.3cm} \tilde{g}} & [\overline{\mathcal{M}}_{0,5}/S_3 \times S_2]&
}
$$
where the vertical arrows induce isomorphisms in the rational Chow ring and rational cohomology, $\tilde{g}: [\overline{\mathcal{M}}_{0,5}/S_3] \to [\overline{\mathcal{M}}_{0,5}/S_3 \times S_2]$ is the quotient map, and the action of $S_2$  symmetrizes the last two marked points. The classes $\mathcal{A}$ and $\mathcal{D}$ are invariant under the action of $S_2$, while the class $\mathcal{B}- \mathcal{C}$ is alternating. This shows in particular that the class $\mathcal{B}- \mathcal{C}$ cannot be in the image of $i^*$ and in particular it cannot be in the image of $f^*$.

So the Proposition is proved if we show that the linear map $f^*$ has rank $2$. Let $p: \overline{II} \to [\overline{\mathcal{M}}_{1,2}/S_2]$ be the map that associates to each bielliptic curve of genus $2$ the corresponding genus $1$ curve with the two branch points. The Chow group of both $[\overline{\mathcal{M}}_{1,2}/S_2]$ and $\overline{\mathcal{M}}_2$ is freely generated by two boundary strata classes, and it is easy to see that the linear map $$p_*\circ f^*: A^1(\overline{\mathcal{M}}_2) \to A^1([\overline{\mathcal{M}}_{1,2}/S_2])$$ is surjective (hence an isomorphism), and in particular this shows that $f^*$ has rank 2.
\end{proof}

\noindent The class $\mathcal{B}- \mathcal{C}$ plays an important role. For this reason it deserves a special name.

\begin{notation} \label{classispeciali} We call $\mathcal{S}$ the class of $\mathcal{B}-\mathcal{C}$ in $A^1(\overline{II})=H^2(\overline{II})$. We call $\mathcal{S}_1, \mathcal{S}_{11}$ the classes in $A^1(\overline{II}_1)$ and $A^1(\overline{II}_{11})$ obtained with the isomorphism of Proposition \ref{aggiunta}. Let now $I_1$, $I_2$ be a partition of $[n]$ in non-empty subsets. After identifying the vector spaces: $$A^*(\overline{II}_{11}^{I_1,I_2})=A^*(\overline{II}_{11})\otimes A^*(\overline{\mathcal{M}}_{0,I_1+1}) \otimes A^*(\overline{\mathcal{M}}_{0,I_2+1})$$ we call $\mathcal{S}^{I_1,I_2}:= (\mathcal{S}\otimes 1\otimes1)$. Analogously we will refer to $\mathcal{S}^{[n]}$ as the class obtained by \virg{adding a rational tail onto $\mathcal{S}_1$}.
\end{notation}

We can then prove the main result of this section, which depends upon Conjecture \ref{getzlerremark} via Corollary \ref{corollariogetzler}.
\begin{theoremstar} (See Conjecture \ref{getzlerremark}.) \label{positivo}  The even Chen-Ruan cohomology ring $H^{ev}_{CR}(\overline{\mathcal{M}}_{2,n})$ is generated, as an algebra over $H^{ev}(\overline{\mathcal{M}}_{2,n})$, by the fundamental classes of the twisted sectors and by the classes $\mathcal{S}, \mathcal{S}^{[n]}, \mathcal{S}^{I_1,I_2}$ (defined in Notation \ref{classispeciali}) for all the possible partitions $\{I_1,I_2\}$ of $[n]$ in non-empty subsets.
\end{theoremstar} 

\begin{proof}  We have proved in Corollary \ref{corollariogetzler}, that for any twisted sector $X$ of $\overline{\mathcal{M}}_{2,n}$, the pull-back map:
$$
f^*:H^{ev}(\overline{\mathcal{M}}_{2,n}) \to H^{ev}(X)
$$
is surjective, unless $X$ is one among $\overline{II}, \overline{II}^{[n]}, \overline{II}^{I_1,I_2}$ . In these cases, we proved in this section, in particular in Proposition \ref{corollariocoker}, that $f^*$ is surjective onto the quotient $H^{ev}(X)/\langle \mathcal{S}^{I_1,I_2}\rangle_{ I_1 \sqcup I_2=[n] }$. As we are adding all the classes $\mathcal{S}, \mathcal{S}^{[n]}, \mathcal{S}^{I_1,I_2}$ as further generators, the theorem is then proved.
\end{proof}
We comment on the optimality of this result:
\begin{remark} \label{negativo} The Chen-Ruan cohomology ring $H^*_{CR}(\overline{\mathcal{M}}_{2,n})$ \emph{strictly} contains the algebra over $H^*(\overline{\mathcal{M}}_{2,n})$ generated by the fundamental classes of the twisted sectors. Also, the even part of the Chen--Ruan cohomology $H^{ev}_{CR}(\overline{\mathcal{M}}_{2,n})$ \emph{strictly} contains the algebra over $H^{ev}(\overline{\mathcal{M}}_{2,n})$ generated by the fundamental classes of the twisted sectors.
The classes $\mathcal{S}^{I_1,I_2}$ cannot be obtained as the Chen--Ruan product of a fundamental class of a twisted sector and a cycle in $\overline{\mathcal{M}}_{2,n}$. It is possible to show by means of lengthy computations that these classes actually do not belong to the algebra generated by the fundamental classes of the twisted sectors.
\end{remark}
We conclude with a couple of considerations. First, the relations among the generators of $H^{ev}_{CR}(\overline{\mathcal{M}}_{2,n})$ are explicitly computable as a consequence of the results of this section and the previous one. However it seems to be very hard to find a concise description of them. 

We believe that some of the odd classes are not in the algebra generated over $H^*(\overline{\mathcal{M}}_{2,n})$ by the fundamental classes. This should follow from the fact that the pull-back does not surject onto the odd cohomology of all the twisted sectors. This in turn would be a consequence of the vanishing of $H^{11}(\overline{\mathcal{M}}_{2,10})$. A proof of this vanishing is not known so far.

\label{sezionepull}


\subsection{The orbifold tautological ring}
 
In this section we give a proposal for an orbifold tautological ring of stable genus $2$ curves, in analogy with what we have proposed in \cite[Section 7]{pagani1} for genus $1$. We develop the theory in the context of Chen--Ruan cohomology. Let the \emph{cohomological tautological ring} $RH^*(\overline{\mathcal{M}}_{g,n})$ be defined as the image of $R^*(\overline{\mathcal{M}}_{g,n})$ in even cohomology under the cycle map. 

\begin{propositionstar}  (See Conjecture \ref{getzlerremark}.) \label{factorization} \label{propositionstar} Let $X$ be a twisted sector of $I(\overline{\mathcal{M}}_{2,n})$, and $f:X \to \overline{\mathcal{M}}_{2,n}$. Then the push-forward map in even cohomology factors through the cohomological tautological ring.
$$\xymatrix{H^{ev}(X) \ar[rr]^{f_*}\ar@{.>}[dr]&& H^*(\overline{\mathcal{M}}_{2,n}) \\
& RH^*(\overline{\mathcal{M}}_{2,n})\ar@{^{(}->}[ur]& }
$$
\end{propositionstar}

\begin{proof} Using Theorem \ref{positivo}, we reduce the claim to proving that the push-forward of the fundamental classes of the twisted sectors, and of the special classes $\mathcal{S}^{I_1, I_2}$, are tautological. The cohomology of $\overline{\mathcal{M}}_{2,n}$ is completely tautological when $n\leq 4$. Indeed, this follows by comparing the Betti numbers of $\overline{\mathcal{M}}_{2,n}$ (see \cite[pp. 20, p.21]{jonas}) with the ranks of the intersection pairings (see \cite[p.11]{stephanie}).  So the push-forwards of $\mathcal{S}, \mathcal{S}_1, \mathcal{S}_{11}$ are tautological. Moreover, the push-forwards of the special classes $\mathcal{S}^{[n]}$ and $\mathcal{S}^{I_1,I_2}$ (see Notation \ref{classispeciali} for their definition) are tautological by the defining property of the tautological ring of being closed under push-forward via natural maps. Thus we are left to show that the push-forwards of the fundamental classes of all the twisted sectors of $\overline{\mathcal{M}}_{2,n}$ are tautological classes. Moreover, again by the closure under push-forward via natural maps, this reduces to showing that the push-forwards of the fundamental classes of the twisted sectors of $\overline{\mathcal{M}}_{2,n}^{NR}$ are tautological. For the twisted sectors that come from the boundary, this follows from the fact that they are constructed by gluing classes in $\overline{\mathcal{M}}_{1,n}$, $n \leq 4$, classes in $[\overline{\mathcal{M}}_{1,n}/S_2]$, $n \leq 6$, and classes in $\overline{\mathcal{M}}_{0,n}$ or $[\overline{\mathcal{M}}_{0,n}/S_2]$ (the cohomology of these spaces is all tautological, see for instance \cite{belo} or \cite{getzler2}).

Finally, if the general element of a twisted sector of $\overline{\mathcal{M}}_{2,n}^{NR}$ is smooth, then either $n \leq 4$ (and in this range we have already seen that the cycles are all tautological), or the twisted sector is either $\overline{\tau}_{11111}$ or $\overline{\tau}_{111111}$ (see Notation \ref{notazionemg}, \ref{notazionecompsmooth}). In these cases, the image is the hyperelliptic locus with $5$ or $6$ of the Weierstrass points marked. The result in this case follows from \cite[Proposition 1]{faberpanda}.
\end{proof}

\noindent This allows us to define:
\begin{definitionstar}  \label{chenruantautolo} We define the \emph{orbifold tautological ring} as:

$$RH^*_{CR}(\overline{\mathcal{M}}_{2,n}):=RH^*(\overline{\mathcal{M}}_{2,n}) \bigoplus_{X \ \textrm{twisted sector}} H^{ev}(X).$$
This is a subring of $H^{ev}_{CR}(\overline{\mathcal{M}}_{2,n})$ as a consequence of Theorem \ref{positivo} and Proposition \ref{propositionstar}.
\end{definitionstar}
 The results of the previous sections, in light of this definition, can be viewed as saying that we have studied generators (and relations) of $RH^*_{CR}(\overline{\mathcal{M}}_{2,n})$ as an algebra over $RH^*(\overline{\mathcal{M}}_{2,n})$.
\begin{corollary} 
The orbifold tautological ring $RH^*_{CR}(\overline{\mathcal{M}}_{2,n})$ is a Poincar\'e duality ring with socle in top degree $3+n$ if and only if (the ordinary) tautological ring $RH^*(\overline{\mathcal{M}}_{2,n})$ is a Poincar\'e duality ring with socle in top degree $3+n$.
\end{corollary}

 To conclude, we make a comment on the tautological stringy Chow ring. A more natural approach to defining $RH^*_{CR}(\overline{\mathcal{M}}_{2,n})$ is by defining $RH^*(X)$ for every $X$ twisted sector of $\overline{\mathcal{M}}_{2,n}$. If $X$ is a twisted sector of $\overline{\mathcal{M}}_{2,n}$, we want to define $R^*(X)$. A possible sensible definition that agrees with our Definition \ref{chenruantautolo}, is as follows. If $X$ is a twisted sector whose general element contains a smooth, genus $2$ curve, then we declare all its rational Chow group to be tautological. If instead $X$ is a twisted sector whose general element is a nodal stable curve, it is obtained by adding rational tails to one of the twisted sectors among Examples \ref{case1}, \ref{case2}, \ref{case3}, \ref{case4}. Again we declare all its rational Chow groups to be tautological, unless $X$ is obtained by adding rational tails to one of the twisted sectors in the first line of Example \ref{case1}. In this case, the coarse moduli space of $X$ is isomorphic to: $$\overline{\mathcal{M}}_{1,k}\times \prod_{i \leq 3} \overline{\mathcal{M}}_{0,k_i}.$$
So, we pose: $$R^*(X):= R^*(\overline{\mathcal{M}}_{1,k}) \times \prod_{i \leq 3} A^*(\overline{\mathcal{M}}_{0,k_i}).$$
Along these lines, one could define $RH^*_{CR}(\overline{\mathcal{M}}_{2,n})$ as the image in the Chen--Ruan cohomology of the tautological stringy Chow ring.


\appendix \section{The inertia stack of $[{\mathcal{M}}_{0,n}/S_2]$ and $[\mathcal{M}_{1,n}/S_2]$}

This section is devoted to the study of the inertia stack (see Section \ref{sectiondefinertia}, and especially Definition \ref{definertia}) of the two stacks $[\mathcal{M}_{0,n}/S_2]$ and $[\mathcal{M}_{1,n}/S_2]$, where $S_2$ is the symmetric group, acting on the first two marked points. It turns out that some of the twisted sectors of these inertia stacks appear as building blocks of some of the twisted sectors of $\overline{\mathcal{M}}_{2,n}$ that come from the boundary (see Definition \ref{bordo}). These twisted sectors are showed in Examples \ref{case1}, \ref{case2}, \ref{case3} and \ref{case4}. Among those building blocks, there are also the compactifications (Definition \ref{compactifiedtwisted}) of some of the twisted sectors of the inertia stack of $\mathcal{M}_{1,n}$. For the inertia stack of $\mathcal{M}_{1,n}$ and $\overline{\mathcal{M}}_{1,n}$, we refer the reader to \cite[Section 3]{pagani1}.

Note first that $[\mathcal{M}_{0,3}/S_2]$ is isomorphic to $B \mu_2$, and therefore its inertia stack is simply two copies of $B \mu_2$ itself. We shall call $0_1^*$ the twisted sector of the inertia stack of $[\mathcal{M}_{0,3}/S_2]$.

\begin{lemma}\label{inertiazero} The inertia stack of $[\mathcal{M}_{0,4}/S_2]$ is:
$$
I([\mathcal{M}_{0,4}/S_2])=([\mathcal{M}_{0,4}/S_2], 1) \sqcup (0_2^*, -1)
$$
where $0_2^*$ is isomorphic to $B \mu_2$.
\end{lemma}

\begin{proof} If $0$, $1$ and $\infty$ are three coordinates on $\mathbb{P}^1$, then there is exactly one further point that is stable under the map $\phi: z \to \frac{1}{z}$: the point $-1$. The twisted sector $0_2^*$ is this configuration of points on $\mathbb{P}^1$, with the automorphism $\phi$.
\end{proof}

Note that $S_2$ acts freely on $\mathcal{M}_{0,n}$, when $n \geq 5$, and therefore the inertia stack of $[\mathcal{M}_{0,n}/S_2]$ coincides with the untwisted sector $[\mathcal{M}_{0,n}/S_2]$ whenever $n \geq 5$. The following corollary can also be obtained as a consequence of Section \ref{rationaltails}, once Lemma \ref{inertiazero} has been established. 

\begin{corollary} The twisted sectors of the inertia stack of $\Big[\overline{\mathcal{M}}_{0,n+2}/S_2\Big]$ are isomorphic to:
$$
B \mu_2 \times \coprod_{L_1 \sqcup L_2=[n]} \overline{\mathcal{M}}_{0,L_1+1} \times \overline{\mathcal{M}}_{0,L_2+1}
$$
where the sets $L_i$ are non-empty.
\end {corollary}
\begin{proof}
It follows from Lemma \ref{inertiazero} that the twisted sectors of the statement are all the possible ways of distributing all the marked points in two sets and gluing them on rational tails on the first and the second stabilized point. We illustrate an example in Figure \ref{figurea1}.
\begin{figure}[ht]
\centering
\psfrag{1}{$0$} 
\psfrag{2}{$1$}
\psfrag{3}{$-1$}
\psfrag{4}{$\infty$}
\includegraphics[scale=0.3]{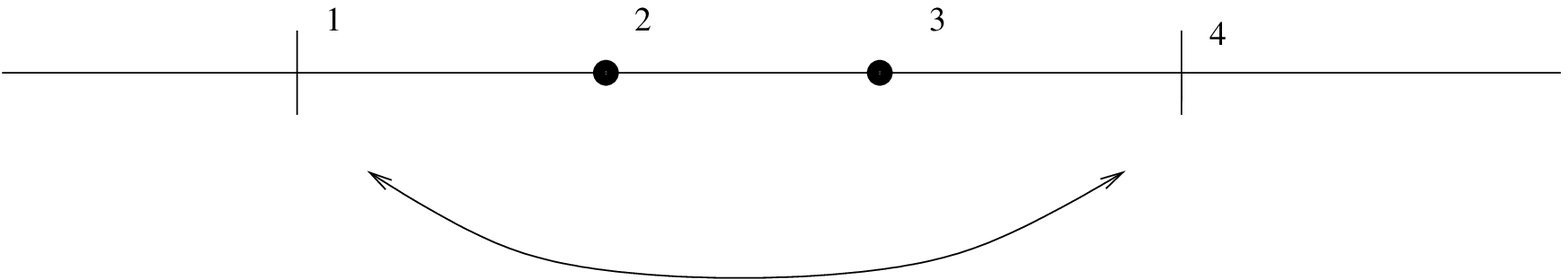} \hspace{0.5cm}
\psfrag{1}{$0$} 
\psfrag{2}{$$}
\psfrag{3}{$$}
\psfrag{4}{$\infty$}
\psfrag{5}{$1$}
\psfrag{6}{$2$} 
\psfrag{7}{$3$}
\psfrag{8}{$4$}
\psfrag{9}{$5$} 
\psfrag{A}{$6$}
\includegraphics[scale=0.3]{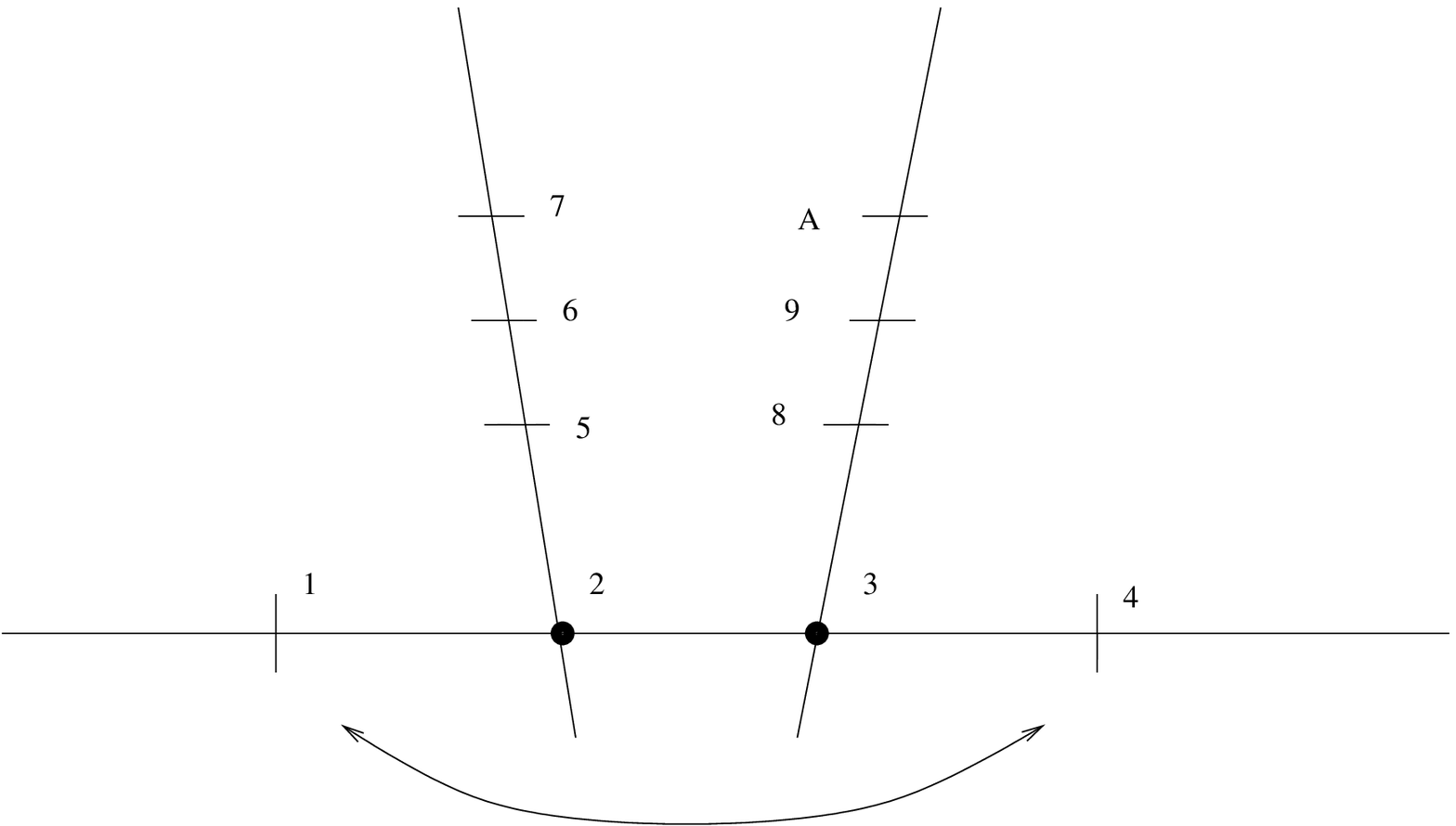}
\caption{Adding marked points on a rational curve, preserving a given automorphism of order $2$.}
\label{figurea1}
\end{figure}
\end{proof}

\begin{definition} \label{inerziazero} We call respectively $0_1^*$ and $0_2^*$ the twisted sectors of the inertia stacks of $[{\mathcal{M}}_{0,3}/S_2]$ and of $[{\mathcal{M}}_{0,4}/S_2]$.
\end{definition}

Now we study the inertia stack of $[\mathcal{M}_{1,n}/S_2]$. We start by observing that for $n \geq 7$ all the objects of the moduli stack $\mathcal{M}_{1,n}$ have trivial stabilizer group, and the action of $S_2$ is free, therefore the inertia stack is trivial. Now if $\phi$ is an automorphism of a genus $1$ curve, then it either exchanges the first two marked points, or it fixes them. Accordingly, the inertia stack splits:

\begin{equation} \label{splitting}
I([\mathcal{M}_{1,n}/S_2])=I^{fix}([\mathcal{M}_{1,n}/S_2])\sqcup I^{exch}([\mathcal{M}_{1,n}/S_2])= [I(\mathcal{M}_{1,n})/S_2] \sqcup I^{exch}([\mathcal{M}_{1,n}/S_2])
\end{equation}

\noindent where $I^{fix}$ is the disjoint union of all the twisted sectors whose distinguished automorphism fixes the two marked points, and $I^{exch}$ are the twisted sectors whose automorphism exchanges the two points. For the twisted sectors of $\mathcal{M}_{1,n}$, see \cite[Section 3]{pagani1}.

Our study of $I^{exch}$ is analogous to the one of Section \ref{inertia2}. Let $(C,x_1, \ldots, x_n)$ be a smooth genus $1$ marked curve. If $\phi$ is an automorphism of it, such that $\phi(x_{n-1})=x_n$ and $\phi(x_n)=x_{n-1}$, we can consider the cyclic covering $\pi: C \to C':= C / \langle \phi \rangle$. The points $x_1, \ldots, x_{n-2}$ correspond one-to-one to points in $C'$ of total branching for the map $\pi$. The last two marked points must be $\pi^{-1}(p)$ for a certain $p$ point of $C'$.

To state the theorem, we define some spaces that are certainly going to be twisted sectors. 

\begin{definition} We define the stack $B_i$ ($2 \leq i \leq 6$) as the moduli stack of double coverings of smooth genus $0$ curves with a reduced branch divisor $D$ of degree $4$, a choice of $1$ marked point $p \notin D$, and a choice of $i-2$ marked points $x_1, \ldots, x_{i-2} \in D$. We denote with $-1$ the automorphism of the double covering.
\end{definition}
\noindent By its very definition, the coarse moduli space of $B_i$ is the same as that of $[\mathcal{M}_{0,5}/S_{6-i}]$. Another twisted sector that we are going to deal with is the following:

\begin{definition} We define the stack $C_2$. It is the moduli stack of connected unramified double coverings of genus $1$ curves, with one marked point on the base (or equivalently, a degree $2$ effective reduced divisor on the fiber, invariant under the automorphism of the covering). In this case, we call again $-1$ the non-trivial automorphism of the covering.
\end{definition}

\noindent The moduli space of $C_2$ is an \'etale triple covering of $\mathcal{M}_{1,1}$.

Following the notation established by us in \cite[Section 3]{pagani1}, we call $(\tilde{C_4},i/-i)$ the two moduli stacks of $\mu_4$-coverings of a genus $0$ curve with a marked point at the only point of non-total branching. 
The moduli stack $(\tilde{C_4},\pm i)$ is the one where the character of the residue action of $\mu_4$ on the cotangent spaces at the two points of total ramification is $\pm i$ respectively. The moduli stacks $(\tilde{C_4'},i/-i)$ (and $(\tilde{C_4''},i/-i)$) are obtained by marking one (respectively two) more points of total branching for the $\mu_6$-covering of the genus $0$ curve. The moduli stacks $(\tilde{C_6},\epsilon/ \epsilon^5)$ and $(\tilde{C_6}',\epsilon/ \epsilon^5)$ are defined in complete analogy.

Therefore one obtains, after \ref{splitting} (see also \cite[Section 3]{pagani1}):

\begin{corollary} The inertia stack of $[\mathcal{M}_{1,n}/S_2]$ for $n \geq 2$ are:
\begin{enumerate}
\item $I\left([\mathcal{M}_{1,2}/S_2]\right)= \left[I \left(\mathcal{M}_{1,2}\right)/S_2 \right]   \coprod (B_2,-1) \coprod (C_2,-1) \coprod (\tilde{C_4}, i/-i) \coprod (\tilde{C_6}, \epsilon/\epsilon^5)$;
\item $I\left([\mathcal{M}_{1,3}/S_2]\right)= \left[I \left(\mathcal{M}_{1,3}\right)/S_2 \right]   \coprod (B_3,-1) \coprod (\tilde{C_4'}, i/-i) \coprod (\tilde{C_6'}, \epsilon/\epsilon^5)$;
\item $I\left([\mathcal{M}_{1,4}/S_2]\right)= \left[I \left(\mathcal{M}_{1,2}\right)/S_2 \right]   \coprod (B_4,-1) \coprod (\tilde{C_4''}, i/-i)$;
\item $I\left([\mathcal{M}_{1,5}/S_2]\right)= \left[I \left(\mathcal{M}_{1,5}\right)/S_2 \right]   \coprod (B_5,-1)$;
\item $I\left([\mathcal{M}_{1,6}/S_2]\right)= \left[I \left(\mathcal{M}_{1,6}\right)/S_2 \right]   \coprod (B_6,-1)$;
\item $n \geq 7$, $I\left([\mathcal{M}_{1,n}/S_2]\right) =([\mathcal{M}_{1,n}/S_2], 1)$.
\end{enumerate}
\end{corollary}
For convenience (in Section \ref{dalbordo}), we define the following subsets of the inertia stacks of $[\mathcal{M}_{1,i}/S_2]$.
$$
\tilde{T_2}:=I^{fix}([\mathcal{M}_{1,2}/S_2])= \{([A_2/S_2],-1),([C_4/S_2],i),([C_4/S_2],-i),(C_6,\epsilon^2),(C_6, \epsilon^4) \},
$$
$$
\tilde{T_3}:= I^{fix}([\mathcal{M}_{1,3}/S_2])= \{([A_3/S_2]), (C_6', \epsilon^2), (C_6', \epsilon^4)\},
$$
$$
\tilde{T_4}:=I^{fix}([\mathcal{M}_{1,4}/S_2])= \{([A_4/S_2]), (C_6', \epsilon^2), (C_6', \epsilon^4)\},
$$
$$
{T_2}^{\rho}:= \{(\tilde{C_4},i),(\tilde{C_4},-i),(\tilde{C_{6}}, \epsilon), (\tilde{C_{6}}, \epsilon^5)\} \subset  I^{exch}([\mathcal{M}_{1,3}/S_2]),
$$
$$
{T_3}^{\rho}:= \{(\tilde{C_4'},i),(\tilde{C_4'},-i),(\tilde{C'_{6}}, \epsilon), (\tilde{C'_{6}}, \epsilon^5)\} \subset  I^{exch}([\mathcal{M}_{1,3}/S_2]),
$$
$$
{T_4}^{\rho}:=  \{(\tilde{C_4''}, i), (\tilde{C_4''}, -i)\} \subset  I^{exch}([\mathcal{M}_{1,4}/S_2]),
$$
$$
\tilde{T_2'}:= \{([C_4/S_2],i),([C_4/S_2],-i),(C_6,\epsilon^2),(C_6, \epsilon^4) \} \subset  I^{fix}([\mathcal{M}_{1,2}/S_2]),
$$
$$
\tilde{T_3'}:= \{ (C_6', \epsilon^2), (C_6', \epsilon^4)\} \subset  I^{fix}([\mathcal{M}_{1,3}/S_2]),
$$
$$
\tilde{T_4'}:= \{(C_6', \epsilon^2), (C_6', \epsilon^4)\} \subset  I^{fix}([\mathcal{M}_{1,4}/S_2]).
$$

\label{appendicea}


\ \\
\ \\
\ \\
\ \\
\ \\

\begin{center}
\textsc{Acknowledgments} 
\end{center}

I would like to thank Carel Faber, Barbara Fantechi and Orsola Tommasi for fundamental suggestions and help in the preparation of this paper. Several of the ideas and results contained here were grown under their influence.

I want to thank Paola Frediani, Fabio Nironi, Dan Petersen, Pietro Pirola, Elisa Tenni, Angelo Vistoli and Stephanie Yang for several helpful conversations and help related to the topics of the present work. I am very grateful to the anonymous referee for comments that led to an improved version of this paper. I am thankful to Susha Parameswaran for useful linguistic suggestions.

During the time of the preparation of this article I was supported by KTH Royal Institute of Technology, and by the Wallenberg foundation. It is a pleasure to acknowledge KTH for the warm atmosphere and the strong support that I received.


\begin{center}
Nicola Pagani\\
Leibniz University of Hannover, Welfengarten 1\\
D-30167, Hannover, Germany.\\
e-mail: npagani@math.uni-hannover.de
\end{center}
\end{document}